\documentclass{article}
\usepackage[intlimits]{amsmath}
\usepackage{amsthm, amssymb}
\usepackage{fancyhdr}
\usepackage{graphicx}
\usepackage{cite}
\usepackage{amssymb}
\newcommand{\bi}{\begin{itemize}}
\newcommand{\ei}{\end{itemize}}
\newcounter{mysection}[section]
\newcounter{mysubsection}[mysection]
\newcounter{mysubsubsection}[mysubsection]
\def\N{\mathbb{N}}
\def\C{\mathbb{C}}
\def\Z{\mathbb{Z}}
\def\SSS{\mathbb{S}}
\def\UU{\mathcal{U}}
\def\RR{\mathcal{R}}
\def\TT{\mathcal{T}}
\def\SS{\mathcal{S}}
\def\KK{\mathcal{K}}
\def\JJ{\mathcal{J}}
\def\g{\mathfrak{g}}
\def\h{\mathfrak{h}}
\def\n{\mathfrak{n}}
\def\b{\mathfrak{b}}
\def\quotient#1#2{%
    \raise1ex\hbox{$#1$}\Big/\lower1ex\hbox{$#2$}%
}

\begin{document}
\theoremstyle{definition}
\newtheorem{mydef}{Definition}[subsection]
\theoremstyle{remark}
\newtheorem{myrmk}[mydef]{\textbf{Remark}}
\theoremstyle{definition}
\newtheorem{mylma}[mydef]{Lemma}
\theoremstyle{definition}
\newtheorem{myex}[mydef]{Example}
\theoremstyle{plain}
\newtheorem{mycor}[mydef]{Corollary}
\theoremstyle{plain}
\newtheorem{mythm}[mydef]{Theorem}
\theoremstyle{definition}

\newtheorem{myprop}[mydef]{Proposition}
\newtheorem*{myproof}{Proof}
\title{Annihilators of simple tensor modules}
\author{Alexandru-Gabriel Sava. Jacobs University Bremen}
\date{January the $1^{st}$, 2012}
\maketitle

\begin{abstract}
We study primitive ideals in the enveloping algebra of finitary locally finite infinite-dimensional complex Lie algebras. In particular we investigate the annihilators of the simple objects in the category of tensor modules. This category has been studied in \cite{PStyr} and \cite{PS}. Moreover we prove that two simple tensor modules are isomorphic if and only if they have the same annihilator. 
\end{abstract}
\newpage

\tableofcontents
\newpage

\section{Introduction}\vspace{0.2cm}

\hspace{0.5cm}Given a Lie algebra $\g$, an important problem in representation theory is to classify all irreducible representations of $\g$ or equivalently, all simple $\g$-modules. This problem proves to be extremely hard even when $\g$ is simple finite-dimensional. This is due to the fact that the class of such irreducible representations is "very large" in general. In fact, $sl(2, \C)$ is the only simple finite-dimensional Lie algebra for which all irreducible representations are classified \cite{Block}. 

The problem of classifying simple $\g$-modules can be turned into an associative algebra problem by passage to the enveloping algebra $\UU(\g)$ : $\g$ is a subspace of $\UU(\g)$ (by the Poincare-Birkhoff-Witt theorem) and every $\g$-module has an unique compatible $\UU(\g)$-module structure. It is also a standard result that a $\g$-module is simple if and only if it is simple as an $\UU(\g)$-module. This approach seems very appealing even if it has the disadvantage of passing to an infinite-dimensional algebra. This is because there is a lot more freedom in doing computations in $\UU(\g)$, and standard associative algebra methods can be applied. 

Even so, this approach is hardly sufficient to solve the problem. Still, one can find "good enough" invariants which reduce our problem substantially. Let $M$ be a $\g$-module. Define $Ann_{\UU(\g)}M=\{x\in\UU(\g) : xm=0\mbox{ for all }m\in M\}$, the annihilator of $M$. The annihilators of simple $\g$-modules are called primitive ideals of $\UU(\g)$. The set of primitive ideals in $\UU(\g)$ is denoted by $Prim(\UU(\g))$. A simple but very useful observation is that two simple modules cannot be isomorphic if they have different annihilators. So rather than looking at the class of simple $\g$-modules (denoted by $\g^{\wedge}$), it is natural to try to understand the structure of $Prim(\UU(\g))$ first. Now our problem seems much more approachable since even if $\g^{\wedge}$ is "large", $Prim(\UU(\g))$ can still be of reasonable size and in some cases completely computable. In addition, as observed by N. Jacobson, one can equip $Prim(\UU(\g))$ with a topology, \cite{Jacobson}. Before going any further, we need to formalize the link between simple $\g$-modules and primitive ideals. 

There is an obvious map $\pi$ from $\g^{\wedge}$ to $Prim(\UU(\g))$ : $M\longmapsto Ann_{\UU(\g)}M$. The map $\pi$ is by definition a surjection but it is hardly an injection. Completing the link between primitive ideals and simple $\g$-modules, is the same as answering the question, what is the fiber of $\pi$ at each point. This question, which probably first appeared in the French school, was given a definitive answer in the case when $\g$ is solvable by J. Dixmier, who gave a complete description of $Prim(\UU(\g))$. In this case, Dixmier proved that primitive ideals are annihilators of a class of induced representations that correspond bijectively to orbits of the action on $\g^*$ of the adjoint group $G$ of $\g$. For the case $\g$ semisimple, which is more relevant for this thesis, M. Duflo proved that when restricted to the category of simple highest weight modules, $\pi$ remains surjective ; moreover Duflo proves that every primitive ideal contains a minimal primitive ideal which he explicitly computes using the work of B. Kostant and the theory of Verma modules. One of Duflo's main results is that such minimal primitive ideals are generated by their intersection with the center of $\UU(\g)$. Even so, the problem of describing the fibers of the restricted map is still very difficult and has been the subject of important research, see \cite{BJ}.   

Our goal in this thesis is to study primitive ideals in $\UU(\g)$ for a finitary locally finite infinite-dimensional Lie algebra $\g$, in particular we are interested in the cases when $\g\simeq sl(\infty, \C)$, $o(\infty, \C)$ and $sp(\infty, \C)$. These are infinite-dimensional Lie algebras (which we will consider over $\C$) exhausted by finite-dimensional simple Lie algebras, and satisfying the condition of finitarity, see \cite{Ba},\cite{BaStr}. The representation theory of these Lie algebras has recently begun to be understood through the work of I. Penkov, V. Serganova, I. Dimitrov, K. Styrkas and others. Interestingly enough, these algebras posses rather peculiar properties that their finite-dimensional siblings do not. For instance, they do not admit any non-trivial finite-dimensional representations. Also, for these Lie algebras, it turns out that the center of $\UU(\g)$ consists only of constants. This theory is of particular interest for us because it raises a first important question. If in the finite-dimensional case every primitive ideal contains a minimal primitive ideal which is centrally generated (and non-trivial !) then are there any non-trivial primitive ideals in $\UU(\g)$ when $\g$ is one of the Lie algebras mentioned above ? After all, the intersection with the center, in this case, cannot be anything but $0$. We will show that the answer to this question is affirmative.\footnote{Primitive ideals in $\UU(\g)$ have been considered in two preprints of the belorussian mathematician Alexei Zhilinskii, see \cite{Zhilinskii1}, \cite{Zhilinskii2}. In fact, more general results are proven in these preprints. Due to language difficulties (the papers are in Russian), we have been unable to understand Zhilinskii's work in full. However, we were able to determine that Zhilinskii does not consider the annihilators of simple tensor modules, while this is precisely what we do. } We will do so by investigating an interesting category of modules introduced in [PStyr], the tensor modules. The simple objects of this category are highest weight modules with respect to a certain choice of the Borel subalgebra and they are "the closest" analogues of finite-dimensional simple modules. It turns out that to distinguish between two simple tensor modules, it is enough to compare their annihilators. In fact even more is true : if a simple $\g$-module $M$ ($\g$ as above) has the same annihilator as a simple tensor module, then $M$ is isomorphic to that tensor module. 

The current thesis is structured in three parts. The first part is concerned with the basic theory of primitive ideals and presents classical results for the finite-dimensional case. The second part is concerned with infinite-dimensional Lie theory. The three Lie algebras : $sl(\infty, \C)$, $o(\infty, \C)$ and $sp(\infty, \C)$ are introduced and the tensor category is presented. Finally, part three is concerned with the results on annihilators of simple tensor modules. 

\section{Primitive ideals in $\UU(\g)$ for a finite-dimensional semisimple Lie algebra $\g$}\
\setcounter{mysection}{2}
We give a general description of primitive ideals in $\UU(\g)$ and present some results of M. Duflo, B. Kostant and J. Dixmier to motivate our study of annihilators of simple tensor modules. Relevant references are : \cite{Dixmier1}, \cite{Dixmier2}, \cite{Dixmier3}, \cite{Dixmier4}, \cite{Duflo1}, \cite{Duflo2}, \cite{Duflo3}, \cite{Humphreys}, \cite{Kostant}, \cite{KoRa}. We begin with basic definitions and some examples and work our way towards Duflo's theorem. Not all results presented will have a proof. However, some relevant statements (those concerning primitive ideals) do come with proof. We follow (as much as possible) the notations in \cite{Dixmier1}. The base field is $\C$ unless stated otherwise. 

\subsection{Ideals. Basic definitions and examples}
\vspace{0.2cm}
In this section $\RR$ denotes an associative ring with unity. Let $I\subset\RR$ be a proper two-sided ideal. \
 
\begin{mydef}\

$I$ is called : 
\begin{enumerate}
\item $\emph{maximal}$ if it is maximal in the set of two-sided ideals of $\RR$ distinct of $\RR$. 
\item $\emph{primitive}$ if it is the annihilator of a simple left $\RR$-module. 
\item $\emph{prime}$ if for any two elements $a,b\in\RR$, the inclusion $a\RR b\subset I$ implies $a\in I$ or $b\in I$. 
\end{enumerate}
\end{mydef}\

\begin{myprop}
The following holds for a two-sided ideal $I\subset \RR$:
$$\emph{maximal} \xrightarrow{(1)}\emph{primitive}\xrightarrow{(2)}\emph{prime}$$
\end{myprop}

\begin{proof} For implication $(1)$ we assume that $I$ is maximal (as a two-sided ideal) and we include $I$ into a maximal left ideal $I_1$ (this can always be done in an associative ring with unity \cite{Hungerford}). Then $\RR/I_1$ is a left $\RR$-module with respect to left multiplication. Moreover, it is clear that $R/I_1$ is simple. Indeed, the preimage of any proper $R$-submodule of $R/I_1$ would be a proper left ideal in $R$ containing $I_1$. 

Let $I'\subset\RR$ be the annihilator of $\RR/I_1$. Then $I'$ is a two-sided ideal in $\RR$ and it is primitive by the previous argument. The inclusion $I\subseteq I_1$ implies that $I$ annihilates $\RR/I_1$. Consequently $I\subseteq I'$, which by maximality of $I$, implies $I=I'$. This proves the first implication. 

For the second implication, let $I$ be the annihilator of a simple left $\RR$-module $M$. Assume for the sake of contradiction that there exist $a,b\in\RR$ such that $a\RR b\subset I$ and $a,b\notin I$. Let $I_1$ and $I_2$ be the two-sided ideals in $\RR$ generated by $a$ and $b$ respectively. Then $I_1I_2\subset I$. Moreover $a,b\notin I$ so $I_1M,I_2M\neq \{0\}$. Then $I_1M=M=I_2M$ since $M$ is simple. Hence $I_1(I_2M)=I_1M=M=(I_1I_2)M\subset IM=0$ false. Conclusion follows. 
\end{proof}
\vspace{0.2cm}

\begin{myrmk}
The converses of the implications in Proposition $2.1.2.$ are false in general as one can see from Example $2.1.6$. However, when $\RR$ is commutative, $(1)$ becomes an equivalence.   
\end{myrmk}
\vspace{0.2cm}

\begin{myprop}
If $\RR$ is commutative, then any primitive ideal is maximal. 
\end{myprop}

\begin{proof} Let $I$ be the annihilator of a simple $\RR$-module $M$. Pick $m\in M,m\neq 0$. Then $\RR m=M$. Define $Ann_{\RR}(m)=\{r\in\RR : rm=0\}$. It is easy to see that $I=Ann_{\RR}(m)$ since $\RR$ is commutative. Now by Zorn's lemma, $I$ lies in some maximal ideal $J\neq\RR$. If $I\neq J$, then $Jm$ is a nontrivial submodule of $M$ so $Jm=M$. Therefore  there exists $j\in J$ such that $jm=m$ and hence $(j-1)m=0$. This implies $j-1\in J$, and consequently $1\in J$, which is false. Thus $I=J$ and $I$ is maximal. 
\end{proof}
\begin{myex}
Let $\RR=\C[X]$. Then by the structure theorem for modules over principal ideal domains, we know that a simple $\RR$-module is isomorphic to $\RR/(\RR p)$ for some irreducible $p\in\RR$, \cite{Hungerford}. In this case, an irreducible element of $\RR$ is a monic degree $1$ polynomial, $p=X-a$ for some $a\in\C$. It's easy to see that $Ann_{\C[X]}\left(\C[X]/(\C[X](X-a))\right)=\C[X](X-a)$ which is clearly a maximal ideal in $\C[X]$. 
\end{myex}
\vspace{0.1cm}

\begin{myex}
Let $\RR=\C[X,Y]$. It is an easy exercise to see that $(y-x-1)$ is a prime ideal. However, this ideal is not maximal (and hence not primitive) since we have the inclusion $(y-x-1)\subset (x-2, y-3)$. 
\end{myex}\vspace{0.2cm}


\subsection{Semisimple Lie algebras}\
\setcounter{mysection}{2}
Let $\g$ be a finite-dimensional Lie algebra with bracket $[\cdot,\cdot]$. For $x\in\g$ define the $\g$-endomorphism $ad_x$ by $ad_x(y)=[x,y]$. Define the following symmetric $\g$-invariant bilinear form on $\g$ : $\KK(x,y)=tr(ad_xad_y)$. $\KK$ is called the Killing form of $\g$ and it plays an important role in the theory of semisimple Lie algebras. 
\vspace{0.2cm}

\begin{mydef}
A Lie algebra $\g$ is said to be semisimple if its Killing form is non-degenerate. 
\end{mydef}
\vspace{0.1cm}
\begin{mydef}
A Lie algebra is said to be simple if $\dim \g>1$ and it has no proper ideals. 
\end{mydef}

\begin{myex}
Let $\g=sl(2, \C)=\{\left(\begin{array}{cc}a&b\\c&d\end{array}\right):a,b,c,d\in\C,a+d=0\}$. Then $\g$ is a simple $3$-dimensional Lie algebra. 
\end{myex}
\vspace{0.2cm}
The following result makes the connection between the definitions above.   

\begin{mythm}
A finite-dimensional Lie algebra is semisimple if and only if it is a direct sum of finite-dimensional simple Lie algebras. 
\end{mythm}

\vspace{0.2cm}
We now present a structure result concerning semisimple Lie algebras. However we need several definitions first.\

\begin{mydef}
Given a semisimple Lie algebra $\g$, a toral subalgebra is defined as a subalgebra for which all elements are semisimple (i.e. $ad_x$ is semisimple). A maximal toral subalgebra of $\g$ is called a Cartan subalgebra. 
\end{mydef}

\begin{myrmk}
$\g$ can be given the structure of a $\g$-module via $x\cdot y=[x,y]$. This is called the adjoint module or the adjoint representation of $\g$.
\end{myrmk} 
\vspace{0.2cm}
From now on, $\g$ will denote a semisimple Lie algebra unless stated otherwise. 

\begin{mydef}\

\begin{enumerate} 
\item Given a $\g$-module $M$ and $\h$ a Cartan subalgebra of $\g$, for $\lambda\in\h^*$, let $M_\lambda=\{m\in M : hm=\lambda(h)m\mbox{ for all }h\in\h\}$. If  $M_\lambda\neq 0$ then $\lambda$ is called a weight of $M$ and $M_\lambda$ is the weight space corresponding to $\lambda$. The set $\{\lambda\in\h^* : M_\lambda\neq\{0\}\}$ is called the support of $M$ and it is denoted by $supp\mbox{ }M$. 
\item The non-zero weights of the adjoint module are called roots and the corresponding weight spaces (denoted by $\g^{\alpha}, \alpha\in\h^*$) are called root spaces. 
\item A module $M$ is called a weight module if $M$ is equal to the sum of its weight spaces.\footnote[1]{It is easy to show that this sum is direct. }
\end{enumerate}
\end{mydef}

Given a Cartan subalgebra $\h\subset\g$, the corresponding set of roots is denoted by $R(\g,\h)$ or simply $R$ if there is no confusion. $R(\g, \h)$ plays an important role in the classification of complex simple Lie algebras. 

The following structure theorem is yet another classical result concerning semisimple Lie algebras. 
\vspace{0.2cm}
\begin{mythm}\cite{Dixmier1}, \cite{Humphreys} \

Let $\g$ be a semisimple Lie algebra and $\h$ be a Cartan subalgebra. Let $R=R(\g,\h)$ and let $\KK$ be the Killing form of $\g$. Then : 
\begin{enumerate}
\item $\g=\h\displaystyle\oplus(\displaystyle\underset{\alpha\in R}\oplus \g^{\alpha})$, and $\dim \g^{\alpha}=1$ for all $\alpha\in R$. 
\item If $\alpha,\beta\in R$, then $[\g^\alpha, g^\beta]\subset \g^{\alpha+\beta}$. If $\alpha\in R$, then $-\alpha \in R$ and $\h_\alpha=[\g^\alpha,\g^{-\alpha}]$ is a one-dimensional subspace of $\h$. It contains one element $H_\alpha$ such that $\alpha(H_\alpha)=2$. This element is called the coroot corresponding to $\alpha$. 
\item If $\alpha+\beta\neq 0$ then $\g^\alpha$ and $\g^\beta$ are orthogonal (with respect to $\KK$). Moreover the restriction of $\KK$ to $\h\times\h$ is non-degenerate, and if $x,y\in\h$, then $\KK(x,y)=\displaystyle\sum_{\alpha\in R}\alpha(x)\alpha(y)$. 
\end{enumerate}
\end{mythm}

\begin{myex}
Let $\g=sl(2, \C)$. $\g$ has a basis $\left\{h=\left(\begin{array}{cc}1&0\\0&-1\end{array}\right),e=\left(\begin{array}{cc}0&1\\0&0\end{array}\right),f=\left(\begin{array}{cc}0&0\\1&0\end{array}\right)\right\}$ with relations $[h,e]=2e$, $[h,f]=-2f$, $[e,f]=h$. Note that $\h=\C h$ is a Cartan subalgebra of $\g$ and $\g$ has $2$ roots : $\alpha$ and $-\alpha$ where $\alpha\in\h^*$ is given by $\alpha(h)=2$. The coroot corresponding to $\alpha$ is $h$ and the coroot corresponding to $-\alpha$ is $-h$.  
\end{myex}
\vspace{0.1cm}
\begin{myprop}
A submodule $N$ of a weight module $M$ is a weight module and $supp\mbox{ }N\subset supp\mbox{ }M$. 
\end{myprop}
\begin{proof}
We will prove the following claim by induction : if $v\in N$ is a vector that can be written as $v=v_1+v_2+\ldots+v_n$ for some $n\in\Z_{\geq0}$ with $v_i\in M_{\lambda_i}\setminus\{0\}, i\in\{1,\ldots,n\}$ and $\lambda_i\neq\lambda_j$ for $i\neq j$ then $v_i\in N_{\lambda_i}$ and $N_{\lambda_i}=N\cap M_{\lambda_i}, i\in\{1,\ldots,n\}$.  

For $n=1$ the claim is obvious. Suppose the claim is true for all $k<n$. Let $v=v_1+v_2+\ldots+v_n\in N$ with $v_i\in M_{\lambda_i}\setminus\{0\}, i\in\{1,\ldots,n\}$. For $h\in\h$ we have  
$$hv=\lambda_1(h)v_1+\ldots+\lambda_n(h)v_n\in N. $$
Hence $hv-\lambda_1(h)v\in N$ and so  
$$(\lambda_2(h)-\lambda_1(h))v_2+\ldots+(\lambda_n(h)-\lambda_1(h))v_n\in N. \mbox{ }\mbox{ }\mbox{ }\mbox{ }\mbox{ }\mbox{ }\mbox{(1)}$$
Note that $(\lambda_k(h)-\lambda_1(h))v_k\in M_{\lambda_k}$. Moreover, since $\lambda_2\neq\lambda_1$, there exists $h_0\in\h$ such that $\lambda_2(h_0)-\lambda_1(h_0)\neq 0$. This implies that not all terms in (1) are $0$ so we may apply the induction hypothesis. Taking into account that $(\lambda_2(h)-\lambda_1(h))v_2\neq 0$, we get $v_2\in N_{\lambda_2}$. Obviously we also get $N_{\lambda_2}=N\cap M_{\lambda_2}$. Applying the induction hypothesis to $v-v_2$ completes the induction step.  

Now let $\tilde{N}$ be the direct sum of all non-trivial $N_\lambda$. Suppose $N\neq\tilde{N}$. Let $v\in N\setminus\tilde{N}$. Since $M$ is weight module, we get $v=v_1+\ldots+v_k$ for some $v_i\in M_{\lambda_i}\setminus\{0\}, i\in\{1,\ldots,k\}$ and $\lambda_i\neq \lambda_j$ for $i\neq j$. But by the previous claim we get $v_i\in N_{\lambda_i}$ and $N_{\lambda_i}=N\cap M_{\lambda_i} $for $i\in\{1,\ldots,k\}$. This implies $v\in\tilde{N}$, which is a contradiction. Hence $N=\tilde{N}$, and so $N$ is a weight module. Moreover, the above argument also shows that $supp\mbox{ }N\subset supp\mbox{ }M$. Thus we are done.   
\end{proof}
The set of roots $R(\g,\h)$ can be partitioned into two disjoint subsets $R_+$ and $R_-$, which we call the set of positive roots and the set of negative roots. These subsets of the partition are closed under addition\footnote[1]{Here closed under addition means that if $\alpha\in R_+$(or $R-$), $\beta\in R_+$(or $R_-$) and $\alpha+\beta\in R(\g, \h)$, then $\alpha+\beta\in R_+$(or $R_-$). } and for any root $\alpha\in R_{\pm}$ we have $-\alpha\in R_{\mp}$. Define : 
$$\n_+=\displaystyle\sum_{\alpha\in R_+}\g^\alpha,\mbox{ }\mbox{ }\mbox{ }\mbox{ }\mbox{ }\mbox{ }\mbox{ }\mbox{ }\n_-=\displaystyle\sum_{\alpha\in R_-}\g^\alpha$$
$$\b_+=\h+\n_+\mbox{ }\mbox{ }\mbox{ }\mbox{ }\mbox{ }\mbox{ }\mbox{ }\mbox{ }\mbox{ }\mbox{ }\mbox{ }\mbox{ }\mbox{ }\b_-=\h+\n_-$$
The subalgebras $\b_+$ and $\b_-$ are called Borel subalgebras\footnote[2]{In fact $\b_-$ is called the opposite Borel subalgebra. } and they will be used later on when discussing Verma modules and highest weight modules. 

The classification of simple finite-dimensional complex Lie algebras goes back to E. Cartan and W. Killing. This classification is in fact a classification of the corresponding root systems. We will not present these results here and point the reader to the influential paper \cite{Dynkin} of E. Dynkin for more details.  
\vspace{0.2cm}
\subsection{The enveloping algebra of a Lie algebra}\

Let $\g$ be a complex Lie algebra with bracket $[\cdot,\cdot]$. Let $\TT(\g)$ be the tensor algebra of $\g$ : 
$$\TT(\g)=\C\oplus\g\oplus(\g\otimes\g)\oplus\ldots$$
Define 
$$\UU(\g)=\TT(\g)/L,$$
where $L$ is the two-sided ideal generated by the elements $x\otimes y-y\otimes x-[x,y]$ with $x,y\in\g$. 
Let $(g_i)_{i\in I}$ be a basis of $\g$. 
\begin{mythm} Poincare-Birkhoff-Witt, \cite{Dixmier1}\

The monomials $g_{i_1}^{r_1}g_{i_2}^{r_2}\ldots g_{i_n}^{r_n}$ with $i_1<\ldots<i_n$ and $r_j\in\Z_{\geq 0},j\in\{1\ldots n\}$ form a basis of $\UU(\g)$
\end{mythm}
\vspace{0.1cm}
\begin{myrmk}
$\g$ is a subspace of $\UU(\g)$. Moreover, for every $\g$-module $M$, there exist an unique compatible $\UU(\g)$-module structure on $M$. 
\end{myrmk}\vspace{0.2cm}

For $u\in\g$ define the maps $L(u),R(u):\UU(\g)\longrightarrow\UU(\g)$ by $v\mapsto uv$ and $v\mapsto vu$ respectively. These maps endow $\UU(\g)$ with the structures of a left and a right $\g$-module. The corresponding representations are called left respectively right regular representations of $\g$ in $\UU(\g)$. Moreover, the map $L-R$ also induces a left action of $\g$ on $\UU(\g)$ and the corresponding representation is called the adjoint representation of $\g$ in $\UU(\g)$. From now on, when using a $\g$-module structure of $\UU(\g)$, we refer to the adjoint representation of $\g$ in $\UU(\g)$.  
\newline

Let $n\in\N$. Let $\UU_n(\g)=\C\oplus\g\oplus(\g\otimes\g)\oplus\ldots(\underbrace{\g\otimes\ldots\otimes\g}_{\mbox{n times}})$, $\UU_0(\g)=\C$, $\UU_1(\g)=\C\oplus \g$, $\UU_n(\g)\UU_p(\g)\subset\UU_{n+p}(\g)$. The sequence $(\UU_n(\g))_{n\geq 0}$ is termed the canonical filtration of $\UU(\g)$. 
\newline

\begin{myrmk}
For all $n\geq 0$, $\UU_n(\g)$ is a finite-dimensional $\g$-submodule of $\UU(\g)$ with respect to the adjoint action of $\g$ on $\UU(\g)$. Given this, it will follow, as a consequence of the following result of H. Weyl, that $\UU(\g)$ is a sum of finite-dimensional simple $\g$-modules. 
\end{myrmk}
\vspace{0.2cm}

\begin{mythm}Weyl's semi-simplicity theorem, \cite{Humphreys}

Let $\g$ be a finite-dimensional semisimple Lie algebra. Then every finite-dimensional $\g$-module is semisimple. 
\end{mythm}

We will now prove the claim of Remark $2.3.3.$ We start with a lemma.  
\begin{mylma}
Let $x_1, x_2, \ldots, x_n\in \g$ and $\sigma$ be a permutation of the set $\{1,\ldots,n\}$. Then $x_1x_2\ldots x_n-x_{\sigma(1)}x_{\sigma(2)}\ldots x_{\sigma(n)}\in \UU_{n-1}(\g)$. 
\end{mylma}
\begin{proof} It is enough to prove the statement in the case $\sigma$ is the transposition $(i\mbox{ }i+1)$, which follows easily from the construction of $\UU(\g)$.  
\end{proof}

The claim of Remark $2.3.3.$ is proved now as follows. Given the adjoint action of $\g$ in $\UU(\g)$, let $M$ be the sum of all simple finite-dimensional submodules of $\UU(\g)$. Suppose $M\neq\UU(\g)$. Pick $u\in\UU(\g)\setminus M$. Obviously $u\neq 0$ and $u$ lies in some $\UU_n(\g)$. Let $N=\UU(\g)u$ be the cyclic module generated by $u$. Using Lemma $2.3.5.$ it is easy to show that $N$ is a submodule of $\UU_n(\g)$. Since $\UU_n(g)$ is finite-dimensional, it follows that $N$ is also finite-dimensional. By Weyl's theorem, we obtain that $N$ is semisimple. But then $N\subset M$ which gives $u\in M$: contradiction ! Thus $M=\UU(\g)$ and so the conclusion follows. \qed

\vspace{0.2cm}
Given the filtered algebra $\UU(\g)$, define the vector spaces $G^n(\g)=\UU_n(\g)/\UU_{n-1}(\g)$ (by convention $\UU_{-1}(\g)=0$) and let $G=\displaystyle\sum_{i\in\Z_{\geq0}} G^i$. The multiplication in $\UU(\g)$ determines, by passing to the quotient, a multiplication in $G$ which makes $G$ an associative algebra with unity. $G$ is termed the graded algebra associated with the filtered algebra $\UU(\g)$.   

\vspace{0.1cm}

Let $S(\g)=\TT(\g)/J$ where $J$ is the two-sided ideal generated by the elements $x\otimes y - y\otimes x$ with $x,y\in\g$. $S(\g)$ is termed the symmetric algebra of $\g$. 
\vspace{0.2cm}
\begin{myrmk}
It follows from Lemma $2.3.5.$ that $G$ is a commutative algebra. Moreover, the canonical injection of $\g$ into $G$ can be uniquely extended to a homomorphism $\theta$ of $S(\g)$ into $G$. It is well known that $\theta$ is an algebra isomorphism \cite{Dixmier1}. This latter statement is usually considered as a part of the Poincare-Birkhoff-Witt Theorem.   
\end{myrmk}
\vspace{0.2cm}
Define the following map $\gamma:S(\g)\longrightarrow\UU(\g)$ : 
\newline
For $x_1, x_2,\ldots, x_n\in \g$ let $\gamma(x_1x_2\ldots x_n)=\displaystyle\frac{1}{n!}\displaystyle\sum_{\sigma\in \SS_n}x_{\sigma(1)}x_{\sigma(2)}\ldots x_{\sigma(n)}$. The map $\gamma$ is termed the canonical bijection\footnote[1]{It is standard result that $\gamma$ is bijective \cite{Dixmier1}. Sometimes $\gamma$ is called the symmetrization map. } of $S(\g)$ onto $\UU(\g)$. 
\vspace{0.2cm}
\begin{myrmk}
$S(\g)$ has a natural $\g$-module structure given by $x\cdot(x_1x_2\ldots x_n)=\displaystyle\sum_{i=1}^nx_1x_2\ldots [x, x_i]\ldots x_n$, and it can be shown that $\gamma$ is a $\g$-module isomorphism.  However, the map $\gamma$ is not an algebra isomorphism. 
\end{myrmk}
\vspace{0.2cm}
From now on, we will identify $G$ with $S(\g)$ via $\theta$. Given the $\g$-module isomorphism $\gamma$, one can translate problems concerning $\UU(\g)$ to problems concerning $S(\g)$. A situation when this idea proves very useful is when trying to determine the structure of the center of $\UU(\g)$.\

Let $Z(\g)$ denote the center of $\UU(\g)$. It will play a vital role for the future results of the thesis. This is mainly due to the following statement. 
\vspace{0.1cm}
\begin{myprop}
Let $M$ be a simple $\g$-module. Then $Z(\g)$ acts on $M$ via scalars. This action defines a homomorphism $\chi:Z(\g)\rightarrow k$ called the central character of $M$. 
\end{myprop}

\begin{proof}

Pick $z\in Z(\g)$. Note that if $m\neq 0$ then $M=\UU(\g)m$. Since $z\in Z(\g)$, then if $zm=\lambda m$ for some $\lambda\in \C$, we have $zx=\lambda x$ for all $x\in M$. It is then enough to prove that there is such an $m_0$ on which $z$ acts by a scalar. 

Consider the $\g$-endomorphism of $M$ given by $m\mapsto zm$. Assume that the endomorphism $\phi : m\mapsto zm$ is algebraic over $\C$. Let $p$ be the minimal polynomial of this endomorphism. Let $\lambda$ be a root of $p$. So $p(X)=(X-\lambda)q(X)$ with $q(\phi)\neq 0$. Then pick an $m\in M$ such that $q(\phi)m\neq 0$. But then note that : $zq(\phi)m=\lambda q(\phi)m$, and by previous considerations we are done. 

Hence all we have to show is that $\phi$ is algebraic over $\C$. Of course, for finite-dimensional $M$ this follows immediately from the Hamilton-Cayley theorem. For our purposes, this will suffice. However, a more general statement is true. In this regard, we present the following lemma.  


\vspace{0.1cm}
\begin{mylma}
Let $A$ be a filtered algebra. Assume that the associated graded algebra $gr(A)$ is finitely generated and commutative. Let $M$ be a simple $A$-module and let $D$ be the set of $A$-endomorphisms of $M$. Then every element $f\in D$ is algebraic over $\C$. 
\end{mylma}

The proof of this lemma is rather irrelevant for this thesis so we point the reader to \cite{Dixmier1} 2.6.4 for a complete argument. Of course, one can see immediately that if we set $A=\UU(\g)$, then the conditions of the lemma are satisfied. Indeed, $M$ is a simple $\g$-module so it is also a simple $\UU(\g)$-module ; $\UU(\g)$ has a natural filtration and the associated graded algebra, $S(\g)$, is obviously abelian and finitely generated by the image of any base of $\g$ through $\gamma^{-1}$ where $\gamma$ is the canonical bijection defined above. 

Thus $\phi$ is algebraic and we are done. \qedhere

\end{proof}


\vspace{0.2cm}
\begin{myrmk}
Modules in which $Z(\g)$ acts via a homomorphism as above are said to admit a central character (usually denoted $\chi$). The converse of Proposition $2.3.1$ is false for obvious reasons, for instance, one can just consider the direct sum of two copies of a simple module $M$.    
\end{myrmk}
\vspace{0.2cm}

\begin{mydef}
Given a $\g$-module $M$, call an element $m\in M$ to be $\g$-invariant if $gm=0$ for all $g\in\g$. Then the set set of $\g$-invariants elements of $M$ form a subspace denoted $M^{\g}$. Clearly $M^{\g}$ is the unique maximal trivial submodule of $M$. 
\end{mydef}
\vspace{0.2cm}

\begin{myrmk}
Taking into account the canonical $\g$-module isomorphism $\gamma$, it is completely clear that $\gamma^{-1} (Z(\g))=S(\g)^{\g}$. Let $Y(\g)=S(\g)^{\g}$. 
\end{myrmk}\vspace{0.2cm}

The structure theories of $Z(\g)$ and that of $\UU(\g)$ as a module over $Z(\g)$ are very important to the study of primitive ideals in $\UU(\g)$. This is perhaps a motivation for the next section which is a short review of these theories. 
 
\vspace{0.2cm}
\begin{mydef}
A $\g$-module $M$ is called a highest weight module if there exist a nonzero weight vector $m$ of some weight $\lambda$ such that $m$ is annihilated by the action of the positive root spaces and $M=\UU(\g)m$. The weight $\lambda$ is called the highest weight of $M$ and $m$ is called a highest weight vector of $M$.  
\end{mydef}

\subsection{The center $Z(\g)$}

We will follow the general structure as presented in \cite{Dixmier1}. However, most of the proofs are skipped since they can all be found in full detail in the book. 

The following lemma shows the non-triviality of $Z(\g)$. 

\begin{mylma}
Let $\{x_i\}_{i\in \{1\ldots n\}}$ be a basis of $\g$. Identify $\g$ and $\g^*$ via a $\g$-invariant bilinear form (say the Killing form)\footnote[1]{For the Killing form, the identification is given by $x\mapsto \KK(x,\cdot)$. } and let $\{y_i\}_{i\in \{1\ldots n\}}\in\g$ be a basis dual to $\{x_i\}_{i\in\{1\ldots n\}}$ with respect to the bilinear form chosen. Let $\phi=\displaystyle\sum\limits_{i=1}^{n} x_iy_i\in\UU(\g)$. Then $\phi\in Z(\g)$. The element $\phi$ is termed the Casimir element (associated to the bilinear form chosen). 
\end{mylma}

\begin{proof}
It is enough to show that $\phi$ commutes with every $x\in\g$. So let $[x,x_i]=\displaystyle\sum_{j=1}^{n}\lambda_{ij}x_j, [x,y_i]=\displaystyle\sum_{j=1}^{n}\mu_{ij}y_j$. Then the assumption that $(x_i)_i$ and $(y_i)_i$ are dual to each other translates into $-\lambda_{ij}=\mu_{ji}$. Indeed, we have $K([x_i, x], y_j)=K(x_i, [x, y_j])$. We get 
$$K(-\displaystyle\sum_{k=1}^n\lambda_{ik}x_k,y_j)=K(x_i,\displaystyle\sum_{k=1}^n\mu_{jk}y_k). $$
Hence
$$-\displaystyle\sum_{k=1}^n\lambda_{ij}K(x_k,y_j)=\displaystyle\sum_{k=1}^n \mu_{jk}K(x_i, y_k). $$
This immediately implies $-\lambda_{ij}=\mu_{ji}$. 
But then  
$$[x,\displaystyle\sum_i x_iy_i]=\displaystyle\sum_i [x,x_i]y_i+\displaystyle\sum_i x_i[x,y_i]=\displaystyle\sum_{ij}\lambda_{ij}x_jy_i+\mu_{ij}x_iy_j=0. $$ 
\end{proof}\vspace{0.2cm}

\begin{myrmk}
Since $\phi\in Z(\g)$, it acts on simple $\g$-modules by a scalar $q\in\C$. This scalar encodes essential information about the module. For instance, for the Lie algebra $sl(2, \C)$, all simple weight modules can be constructed using $q$ and another parameter $v_{hw}\in\C/\Z$. 
\end{myrmk}
\vspace{0.2cm}

We now turn our attention to the structure of $Z(\g)$. Recall that $Z(\g)$ corresponds to the subalgebra $Y(\g)=S(\g)^\g$ in $S(\g)$, through the $\g$-module isomorphism $\gamma$. This argument, along with the identification of $\g$ with $\g^*$ as above, redirects our focus towards the structure of $S(\g^{*})^\g$. To this end, we present the following result.

\begin{mylma}
Let $\pi$ be a finite-dimensional representation of $\g$, and let $m\in \N$, The function $f_{\pi, m}: x\mapsto tr(\pi(x)^m)$ on $\g$ belongs to $S(\g^{*})^\g$.
\end{mylma}

\begin{proof}
If $S^k(\g^{*})$ denotes the span of the symmetric tensors of degree $k$, then $f_{\pi, m}\in S^m(\g^{*})$. Any element of $S^k(\g^*)$ is a symmetric $m$-linear form. We have  
$$f_{\pi, m}(x_1,\ldots,x_m)=\displaystyle\frac{1}{m!}\displaystyle\sum_{\sigma\in \SS_m}tr(\pi(x_{\sigma(1)})\ldots\pi(x_{\sigma(m)})), $$
where here $\SS_k$ denotes the $k^{th}$ symmetric group. Hence  
$$m!(xf_{\pi, m})(x_1,\ldots,x_m)=\displaystyle\sum_{\sigma\in\SS_m}\displaystyle\sum\limits_{i=1}^{m}tr(\pi(x_{\sigma(1)})\ldots[\pi(x),\pi(x_{\sigma(i)})]\ldots\pi(x_{\sigma(m)}))$$
$$\mbox{ }\mbox{ }\mbox{ }\mbox{ }\mbox{ }\mbox{ }\mbox{ }\mbox{ }\mbox{ }\mbox{ }\mbox{ }\mbox{ }\mbox{ }\mbox{ }\mbox{ }\mbox{ }\mbox{ }\mbox{ }\mbox{ }\mbox{ }\mbox{ }\mbox{ }\mbox{ }\mbox{ }\mbox{ }\mbox{ }\mbox{ }\mbox{ }\mbox{ }\mbox{ }\mbox{ }\mbox{ }\mbox{ }\mbox{ }\mbox{ }\mbox{ }\mbox{ }\mbox{ }\mbox{ }\mbox{ }\mbox{ }\mbox{ }\mbox{ }\mbox{ }\mbox{ }\mbox{ }=\displaystyle\sum_{\sigma\in \SS_m}tr(\pi(x)\pi(x_{\sigma(1)})\ldots\pi(x_{\sigma(m)}))-tr(\pi(x_{\sigma(1)})\ldots\pi(x_{\sigma(m)})\pi(x))=0.$$
This gives $xf_{\pi, m}=0$ thus ending the proof. 
\end{proof}

\begin{myex}
Let $\g=sl(2, \C)$ with the standard basis $e,f,h$. For $m=1$ $tr(\pi(\cdot))=0$. For $m=2$, let $\pi$ be the natural representation i.e. $\pi$ is the inclusion $sl(2, \C)\longrightarrow gl(2, \C)$. Then $f_{\pi, m}$ is $x\mapsto tr(\pi(x)^2),\hspace{0.2cm}x\in\g$, and the corresponding symmetric bilinear form is $(x,y)\mapsto\displaystyle\frac{1}{2}( tr(\pi(x)\pi(y))+tr(\pi(y)\pi(x)))=\displaystyle\frac{1}{2}(tr(xy)+tr(yx))=tr(xy)\hspace{0.2cm} x, y\in\g$. 

For $x\in\g$ set $x=a_xh+b_xe+c_xf\hspace{0.2cm}a_x, b_x, c_x\in\C$. An easy computation shows that $f_{\pi, m}(x)=2(a_x^2+b_xc_x)$ and that : $\KK(h, x)=8a_x,\hspace{0.2cm}\KK(e, x)=4c_x,\hspace{0.2cm}\KK(f, x)=4b_x$. Hence $f_{\pi, m}=\displaystyle\frac{1}{32}(\KK(h, \cdot)^2 + 4\KK(e, \cdot)\KK(f, \cdot))$. Passing back in $S(\g)^{\g}$, we get an element $\tilde{f}_{\pi, m}$ given by $\tilde{f}_{\pi, m}=\displaystyle\frac{1}{32}(\tilde{h}^2+4\tilde{e}\tilde{f})$ where $\tilde{h}, \tilde{e}$ and $\tilde{f}$ are the images in $S(\g)$ of the standard basis of $\g$ through $\gamma^{-1}$. The image of this element in $\UU(\g)$ is then  
$$\bar{f}_{\pi, m}=\displaystyle\frac{1}{32}(h^2+2(ef+fe))\in Z(\g). $$
This is (up to scaling) the standard Casimir element in $Z(\g)$. 
\end{myex}
\newpage
\begin{mydef}\

\begin{enumerate}
\item Let $\g$ be a semisimple Lie algebra and $\h$ a splitting Cartan subalgebra. Let $\alpha$ be a root of $\g$ and $H_\alpha$ be the corresponding coroot. Define the reflection $s_\alpha:\h^{*}\longrightarrow \h^*, s_\alpha(\lambda)=\lambda-\lambda(H_\alpha)\alpha$. Note that $s_\alpha^2=1$ so $s_\alpha$ is an automorphism of $\h^{*}$ and moreover $s_\alpha$ preserves the bilinear form $K(\cdot,\cdot)$ on $\h^*$. 
\item Let $W$ be the group of automorphisms of $\h^*$ generated by $s_\alpha$ for all roots $\alpha$ of $\g$. $W$ is termed the Weyl group of $\g$. 
\end{enumerate}
\end{mydef}
\vspace{0.1cm}
\begin{myex}
Let $\g=sl(2, \C)$. Then $\h^{*}$ is one dimensional, hence there is only one reflection : $\lambda\mapsto -\lambda$. Consequently, $W$ is isomorphic to $\Z/2\Z$. 
\end{myex}
\vspace{0.1cm}

Denote the set of polynomial functions $f\in S(\h^{*})$ for which $wf=f$ for all $w\in W$ (i.e. $f$ is $W$ invariant) by $S(\h^{*})^W$, and the subset of $W$ invariant elements which are homogeneous of degree $m$ by $S^m(\h^{*})^W$. 
 
The role of $S(\h^{*})^W$ is not at all obvious but the following theorem shows that this subspace is closely connected to $S(\g^{*})^\g$ which, as noted earlier, is relevant for the structure of $Z(\g)$. It is a classical result that every element of $S(\h^{*})^W$ is a linear combinations of functions as in Lemma $2.4.3.$. This in turn gives an explicit set of generating vectors for the subspace $S(\g^{*})^\g$.   
\vspace{0.1cm}
\begin{mythm}\cite{Bourbaki}, \cite{LS}\

Let $i:S(\g^{*})\longrightarrow S(\h^{*})$ be the restriction homomorphism. 
\begin{enumerate}
\item The mapping $i_{|S(\g^{*})^\g}$ is an isomorphism between $S(\g^{*})^\g$ and $S(\h^{*})^W$. 
\item For $m\in\N$ the space $S^m(\g^{*})^\g$ is the set of linear combinations of functions of the form $x\mapsto tr(\pi(x)^m)$, where $\pi$ runs over all finite-dimensional representations of $\g$. 
\end{enumerate}
\end{mythm}

\begin{myex}
Let $\phi=h^2+2(ef+fe)$ and consider its image in $S(\g^{*})$ i.e. the function $g(x)=\KK^2(h,x)\KK(h,x)+4\KK(e,x)\KK(f,x)$. Then $W$ acts on $\h^{*}$ by $\lambda\mapsto -\lambda$ so on $\h$ by $h\mapsto -h$. The restriction of the above function is clearly invariant under the action of $W$. 
\end{myex}

The following statement completes the structure description of the algebra $Z(\g)$. The first two parts are results concerning $Y(\g)$, which we recall to be the preimage of $Z(\g)$ in $S(\g)$ under the canonical $\g$-module isomorphism, while the third is the main result concerning $Z(\g)$. Note that $W$ acts on $\h^{*}$ and, by transport of structure via the isomorphism $\h^*\simeq \h$, it also acts on $\h$.

\begin{mythm}\cite{Dixmier1}\

\begin{enumerate}
\item Let $S(\h)^W$ be the set of $W$-invariant elements in $S(\h)$. Then the algebras $Y(\g)$ and $S(\h)^W$ are isomorphic. 
\item If $r$ is the rank of $\g$ (the dimension of a Cartan subalgebra) then there are $r$ algebraically independent elements in $S(\h)^W$ (and thus $r$ in $Y(\g)$) that generate the algebra $S(\h)^W$ (and thus $Y(\g)$). 
\item The algebra $Z(\g)$ is isomorphic to a polynomial algebra in $r$ indeterminates. 
\end{enumerate}
\end{mythm} 

We conclude this section with an important result due to Kostant which concerns the structure of $\UU(\g)$ as a module over $Z(\g)$. \

\begin{mythm} \cite{Kostant}\

 Let $Y_+(\g)$ be the subset of $Y(\g)$ consisting of the elements without constant term and let $H(\g)$ be a graded complement of the ideal $Y_+(\g)S(\g)$. Denote by $K(\g)$ the image of $H(\g)$ in $\UU(\g)$ under the canonical $\g$-module isomorphism. Then the linear mapping $g:K(\g)\otimes Z(\g)\longrightarrow \UU(\g), g(x,z)=xz$ is a vector space isomorphism, or equivalently, $\UU(\g)$ is a free $Z(\g)$-module. 
\end{mythm}

\subsection{Verma modules}\

Now that we have set up the preliminaries for semisimple Lie algebras and their enveloping algebra, we turn our attention to a class of $\g$-modules ($\g$ is assumed semisimple) which was first introduced by Harish-Chandra in the late $1940$'s. This section is devoted to the study of these modules and some of their properties. 
 
 Let $\g$ be a semisimple Lie algebra and $\h$ a Cartan subalgebra. Let $R(\g, \h)$ be its root system and let $W$ be the Weyl group. Choose $B$ to be a basis for $R$. Denote by $R_+$ and $R_-$ the set of positive and negative roots of $\g$. Recall that  
$$\n_+=\displaystyle\sum_{\alpha\in R_+}\g^\alpha,\mbox{ }\mbox{ }\mbox{ }\mbox{ }\mbox{ }\mbox{ }\mbox{ }\mbox{ }\n_-=\displaystyle\sum_{\alpha\in R_-}\g^\alpha, $$
$$\b_+=\h+\n_+, \mbox{ }\mbox{ }\mbox{ }\mbox{ }\mbox{ }\mbox{ }\mbox{ }\mbox{ }\mbox{ }\mbox{ }\mbox{ }\mbox{ }\mbox{ }\b_-=\h+\n_-, $$
and set $\rho$ to be the half-sum of the positive roots. 

\begin{mydef}

In the notations above $\UU(\g)$ has a right $\UU(\b_+)$-module structure via multiplication on the right. For $\lambda\in\h^{*}$, let $\C_\lambda$ be the one-dimensional left $\b_+$-module structure defined by $\lambda-\rho$ (i.e. $\n_+$ acts trivially and the $\h$ acts via the homomorphism $\lambda-\rho:\h^*\longrightarrow \C$). Since $\UU(\g)$ is also a left $\g$-module, then the following tensor product  
$$M(\lambda)=\UU(\g)\displaystyle\otimes_{\UU(\b_+)}\C_\lambda$$
is also an $\UU(\g)$-module. This module is called a Verma module (with respect to the chosen Borel subalgebra)
\end{mydef}
\vspace{0.2cm}

For $\alpha\in R$, choose $X_\alpha\in \g^\alpha, X_\alpha\neq 0$. Set $\alpha_1,\ldots,\alpha_n$ to be the roots in $R_+$ and $H_1,\ldots,H_r$, a basis of $\h$. The following proposition shows that $M(\lambda)$ is a highest weight module. 

\begin{myprop}\cite{HC2}, \cite{Verma}

\begin{enumerate}
\item $M(\lambda)=\displaystyle\bigoplus_{\mu\in\h^{*}}M(\lambda)_\mu$. 
\item $M(\lambda)_\mu=\displaystyle\sum\limits_{p_i\in\N, \lambda-\rho-\mu=\sum p_i\alpha_i}X_{-\alpha_1}^{p_1}\ldots X_{-\alpha_n}^{p_n}\otimes\C_\lambda$  and so $dim$ $M(\lambda)_\mu=\#\{(n_\alpha)_{\alpha\in B} : n_\alpha\in\N, \displaystyle\sum\limits_{\alpha\in B}n_\alpha\alpha=\lambda-\rho-\mu\}$. 
\item $M(\lambda)$ is a highest weight module and 
$$M(\lambda)_{\lambda-\rho}=1\otimes\C_\lambda,\mbox{ }\mbox{ }\mbox{ } M(\lambda)=\UU(\n_-)M(\lambda)_{\lambda-\rho}, \mbox{ }\mbox{ }\mbox{ }\n_+M(\lambda)_{\lambda-\rho}=0.$$ 
\end{enumerate}
\end{myprop}\

\begin{myrmk}
It is a well known result that every highest weight module $V$ is isomorphic to a quotient of a Verma module. This implies that the highest weight space of any highest weight module is one-dimensional. 
\end{myrmk}
\vspace{0.1cm}
\begin{mycor}
Every highest weight module $V$ admits a central character. 
\end{mycor}

\begin{proof}
Pick $z\in Z(\g)$. Then since $V$ is cyclic, it is enough to prove that $z$ acts by a scalar on a highest weight vector $v_{hw}\neq 0$. Indeed, if $zv_{hw}=cv_{hw}$ then for every $v\in V$ we have $v=xv_{hw}$ for some $x\in\UU(\g)$. But then $zv=z(xv_{hw})=x(zv_{hw})=x(cv_{hw})=c(xv_{hw})=cm$.  

Now for all $h\in\h$ we have have  
$$h(zv_{hw})=z(hv_{hw})=z(\lambda(h)v_{hw})=\lambda(h)(zv_{hw}).$$
This means that $zv_{hw}$ is in the highest weight space. But the highest weight space is one-dimensional by Remark $2.5.3.$ Hence $zv_{hw}=cv_{hw}$ for some $c\in\C$ which, according to the above argument, completes the proof. 
\end{proof}
\begin{myprop}
\vspace{0.1cm}
$M(\lambda)$ has a unique simple quotient denoted by $L(\lambda)$. 
\end{myprop}

\begin{proof} 
Let $N$ be a submodule of $M(\lambda)$. By Proposition $2.5.2.$ $M(\lambda)$ is a weight module. Taking into account Proposition $2.2.10.$ we get that $N$ is also a weight module. Moreover,  
$$N=\displaystyle\sum_{\mu\in\h^*}(N\cap M(\lambda)_\mu). $$
Note that if $N$ intersects $M(\lambda)_{\lambda-\rho}$ non-trivially then $N=M(\lambda)$ since $M(\lambda)_{\lambda-\rho}$ is one-dimensional and generates $M(\lambda)$. This implies that the sum $S$ of all submodules of $M(\lambda)$ distinct from $M(\lambda)$ is contained in the subspace $M'(\lambda)=\displaystyle\sum_{\mu\neq\lambda-\rho}M(\lambda)_\mu$.

It is obvious that $S$ is the largest submodule of $M(\lambda)$ distinct from $M(\lambda)$. This observation implies directly that $M(\lambda)/S$ is simple and that any other simple quotient is isomorphic to $M(\lambda)/S$. 
\end{proof}
The modules $L(\lambda)$ are simple highest weight modules and all simple highest weight modules arise from this construction. Indeed, the image of any highest weight vector of $M(\lambda)$ through the quotient will be a highest weight vector for $L(\lambda)$. Moreover, if a simple module is a highest weight module of highest weight $\lambda-\rho\in\h^*$, then it is a quotient of the Verma module $M(\lambda)$ and so, by Proposition $2.5.2.$, is isomorphic to $L(\lambda)$.  

Since $M(\lambda)$ is a highest weight module, it admits a central character. Denote the central character of $M(\lambda)$ by $\chi_\lambda$. The following two propositions concern properties of $\chi_\lambda$. Proofs and more details can be found in \cite{HC2} and \cite{Verma}. 

\begin{myprop}

$\chi_\lambda=\chi_{\lambda'}$ if and only if $\lambda'\in W\lambda$. 
\end{myprop}

\begin{myprop}

Let $\chi$ be a homomorphism of $Z(\g)$ into $\C$. Then there exist $\lambda\in \h^*$ such that $\chi=\chi_\lambda$. 
\end{myprop}

The last part of this section is concerned with the computation of the annihilators of the Verma modules. This is then used in the next section when we present M. Duflo's result. We need a preliminary lemma. 

\begin{mylma}

We retain the notation in Theorem $2.4.10.$ Let $\beta$ be the canonical map from $S(\g)$ to $\UU(\g)$. Let $K(\g)^{\n_-}$ be the $\n_-$ invariant elements of $K(\g)$, $\lambda\in\h^*$, and $J$ the left ideal generated by $\n_+$ and $h-\lambda(h)$, $h\in\h$.  
\begin{enumerate}
\item $\UU(\g)=J\oplus \UU(\n_-)$. 
\item Denote by $Q$ the projection of $\UU(\g)$ onto $\UU(\n_-)$ in the above decomposition. If $u\in K(\g)^{\n_-}$ and $u\notin \UU_p(\g)$, then $Q(u)\notin\UU_p(\n_-)$. 
\end{enumerate}

\end{mylma}

\begin{proof}

Let $V$ be a simple highest weight module with highest weight $\lambda$ and $v_{hw}$ a highest weight vector. Then it is a known result that the annihilator of $v_{hw}$ in $\UU(\g)$ is $J$, see for instance \cite{Higman}, \cite{Dixmier1}. This observation combined with the fact that $V=\UU(\n_-)v_{hw}$ implies that $\UU(\g)$ is a sum of $J$ and $\UU(\n_-)$. The fact that the sum is direct is trivial. This concludes the first part.

For the second part we refer the reader to \cite{Dixmier1}. 
\end{proof}
\vspace{0.2cm}

\begin{myprop}\cite{Dixmier3}, \cite{Duflo3}

Let $\g$ be a semisimple Lie algebra, $\h$ a Cartan subalgebra and $\lambda\in \h^*$. The annihilator of $M(\lambda)$ is $\UU(\g)Ker\mbox{ }\chi_\lambda$. 
\end{myprop}

\begin{proof}

Again we retain the notation in Theorem $2.4.10.$ First, it is well known that $H(\g)$ is semisimple as a $\g$-module and moreover, its simple constituents are finite-dimensional,\cite{Kostant}. Also, the $\g$-action on the $H_i$ preserves the degree for any monomial in $H_i$. We get $H(\g)=\displaystyle\oplus_{i\in I}H_i$ where the $H_i$ are finite-dimensional simple modules and their elements are all homogeneous of degree $n_i\in\N$. Let $K_i=\beta(H_i)$.  Since $K_i$ are finite-dimensional, they are lowest weight modules (i.e. highest weight with respect to the opposite Borel subalgebra) so we can pick $t_i\in K_i$ such that $\n_-t_i=0$. These $t_i$ are well-defined up to a scalar (since they are in the lowest weight space). Set $(z_\gamma)_\gamma$ a basis for $Z(\g)$, $J_\lambda$ to be the annihilator of $M(\lambda)$ and $A=\UU(\g)Ker\mbox{ }\chi_\lambda$. 

It is completely clear that $A\subset J_\lambda$. Suppose $J_\lambda\setminus A\neq \emptyset$. Both $A$ and $J_\lambda$ are submodules of $\UU(\g)$ with respect to the adjoint action. Now $\UU(\g)$ is a sum of simple finite-dimensional submodules by Remark $2.3.2.$ These observations lead us to the conclusion that there exist a finite-dimensional simple submodule $N$ of $J_\lambda$ such that $N\cap A=\{0\}$. Pick a nonzero vector $\n_-$-invariant element $u$ from $N$. In light of Theorem $2.4.10$ we get that there exist $l_\gamma\in K(\g)$ so that 
$$u=\displaystyle\sum_\gamma l_\gamma z_\gamma\mbox{ }\mbox{ }\mbox{ }\mbox{ }\mbox{ }\mbox{ }\mbox{ } (2).$$
Since $u$ is $\n_-$ invariant we get that for all $n\in\n_-$  
$$0=[n,u]=[n,\displaystyle\sum_\gamma l_\gamma z_\gamma]=\displaystyle\sum_\gamma [n,l_\gamma]z_\gamma, $$
which gives $[n,l_\gamma]=0$ for all $\gamma$. Hence $l_\gamma\in K(\g)^{n_-}$ for all $\gamma$. By previous considerations, we obtain that $l_\gamma$ is a linear combination of some $t_i$.  Plugging this in (w) and factoring the $t_i$s gives us that $u=\displaystyle\sum_i t_i z'_i$ where $z'_i\in Z(\g)$. If $v_{hw}$ is the canonical generator of $M(\lambda)$ we get 
$$0=uv_{hw}=\displaystyle\sum_i t_i z'_i m=\displaystyle\sum_i \chi_\lambda(z'_i)t_i m. $$
This implies that $\displaystyle\sum_i \chi(z'_i)t_i$ is in the annihilator of $v_{hw}$ and also in $K(\g)^{\n_-}$. In light of Lemma $2.5.8.$ we obtain 
$$\displaystyle\sum_i \chi(z'_i)t_i=0. $$
Hence $\chi(z'_i)=0$ for all $i$ and so $z'_i\in Ker\mbox{ }\chi_\lambda$. This implies that $u\in A$, contradiction. Conclusion follows. 
\end{proof}. 


\subsection{Duflo's result}\

There are two main results with which we conclude this first chapter. Both are due to M. Duflo. In the last part of the section we present an application of these results for $\g=sl(2, \C)$. 

\begin{mythm} Weak form of Duflo's theorem, \cite{Duflo1}

Let $\g$ be a semisimple Lie algebra, $\h$ a Cartan subalgebra and $W$ its Weyl group. For every $\lambda\in\h^*$, let $J_\lambda$ be the annihilator of $M(\lambda)$. Then :
\begin{enumerate}
\item The ideals $J_\lambda$ are primitive. 
\item $J_\lambda=J_{\lambda '}$ if and only if $\lambda '\in W\lambda$. 
\item The ideals $J_\lambda$ are the minimal primitive ideals of $\UU(\g)$. 
\end{enumerate}
\end{mythm}
\begin{proof}
The second part is a direct consequence of Proposition $2.5.6.$ and Proposition $2.5.9.$ Indeed, note that if $\chi_\lambda\neq\chi_{\lambda'}$ then there exists $z\in Z(\g)$ so that $\chi_\lambda(z)=c_1\in\C$, $\chi_{\lambda'}(z)=c_2\in\C$ and $c_1\neq c_2$. But then by Proposition $2.5.9.$ we get $z-c_1\in ker\mbox{ }\chi_\lambda\subset J_\lambda=J_{\lambda'}$. So $z-c_1\in ker\mbox{ }\chi_{\lambda'}$. This implies $c_2-c_1=0$, contradiction. Hence $\chi_\lambda=\chi_{\lambda'}$ and then the conclusion of the second part follows by Proposition $2.5.9.$ Now the first part follows since it is well known (see \cite{Dixmier1}) that for every $\lambda\in\h^*$ there exist $\mu\in\h^*$ and $w\in W$ such that $M(\mu)$ is simple and $\mu=w\lambda$. 

For the third part, let $J$ be some primitive ideal. It annihilates a simple $\g$-module $V$. $V$ admits a central character, which by Proposition $2.5.7.$ is equal to $\chi_\lambda$ for some $\lambda\in\h^*$. It follows easily that $J_\lambda\subset J$. Indeed, we immediately get $ker\mbox{ }\chi_\lambda\subset J$ and so the two-sided ideal generated but this kernel lies in $J$. So if $J$ is minimal then $J=J_\lambda$. This argument shows that if $J$ is some minimal primitive ideal then it is equal to some $J_\lambda$. It remains to show that all $J_\lambda$ are minimal. 

Suppose $J_\lambda$ contains some primitive ideal $J'$. Then as before, $J'$ contains some $J_\mu, \mu\in\h^*$. But then $J_\mu\in J_\lambda$ implies $Ker\mbox{ }\chi_\mu\subset Ker\mbox{ }\chi_\lambda$ and since both these ideals are maximal in $Z(\g)$ we get that they are equal so $J_\mu=J_\lambda$. This ends the argument.  
\end{proof}
\vspace{0.2cm}
\begin{mythm} Strong form of Duflo's theorem, \cite{Duflo1}

If $L(\lambda)$ denotes the unique simple quotient of $M(\lambda)$, then every primitive ideal is the annihilator of some $L(\lambda)$. 
\end{mythm}
\vspace{0.2cm}
We do not present the proof of this theorem as it is rather lengthy and technical. The following theorem combines the results of Duflo and gives an overview of $Prim(\UU(\g))$.  

\begin{mythm} \cite{Dixmier1}

Let $\g$ be a semisimple Lie algebra and $\h$ a Cartan subalgebra. For $\lambda\in\h^*$ set $J_\lambda=\UU(\g)Ker\mbox{ }\chi_\lambda$. 
\begin{enumerate}
\item The set of primitive ideals of $\UU(\g)$ containing $J_\lambda$ is finite, possesses a largest element denoted $J_\lambda '$ and if $\lambda(H_\alpha)\notin\Z\setminus\{0\}$ for every root $\alpha$, then $J_\lambda=J_\lambda '$. \footnote[1]{$H_\alpha$ is the coroot defined in Theorem $2.2.2.$} 
\item If $\lambda(H_\alpha)\notin\Z_-$ for every positive root $\alpha$, then $J_\lambda '$ is the annihilator of $L(\lambda)$. 
\end{enumerate}
\end{mythm}
\vspace{0.2cm}

As an illustration of the strength of Duflo's results,we show how they imply a complete description of primitive ideals in $\UU(sl(2, \C))$. We need several preliminary statements. 

\begin{myprop}\cite{BGG1}, \cite{BGG2}, \cite{Verma} 

$M(\lambda)$ is simple if and only if for every positive root $\alpha$, we have $\lambda(H_\alpha)\notin\Z_{>0}$. 
\end{myprop}\vspace{0.2cm}
\begin{mylma}
$M(\lambda)$ and $L(\lambda)$ have the same central character, in particular $J_\lambda\subset Ann_{\UU(\g)}L(\lambda)$. 
\end{mylma}
\begin{proof}
Let $v_{hw}$ be the canonical generator of $M(\lambda)$ and $v_{hw}'$ a generator of $L(\lambda)$. Pick $z\in Z(\g)$ arbitrarily. It is easy to check $zv_{hw}$ is a highest weight vector of $L(\lambda)$ and hence $zv_{hw}=cv_{hw}'$ for some $c\in\C$. Moreover, since the mapping $v\mapsto zv$ is a $\g$-module homomorphism, and since $z$ was chosen arbitrarily, the conclusion follows.  
\end{proof}
\vspace{0.2cm}
Now set $\g=sl(2, \C)$. Let $V$ be the natural module\footnote[2]{The corresponding representation is given by the inclusion $sl(2, \C)\longrightarrow gl(2, \C). $} of $\g$. Then it is well known fact that the list $\C, V, S^2V,\ldots$ forms a complete list of finite-dimensional simple $\g$-modules. \footnote[3]{Here $\C$ denotes the one dimensional trivial module and $S^kM$ denotes the $k^{th}$ symmetric tensor power of the vector space $M$. }

\begin{myprop}

Set $\h=\C h$ where $\{h,e,f\}$ is the standard basis of $\g$. Let $J$ be a primitive ideal of $\UU(\g)$. Then either $J=J_\lambda$ for some $\lambda\in\h^*$ or $J=Ann_{\UU(\g)}S^kV$ for some $k\in\Z_{\geq 0}$ where by convention $S^0V=\C$ and $S^1V=V$. 
\end{myprop}
\begin{proof}
In this case there is only one positive root $\alpha$ and $h$ is the corresponding coroot. By Theorem 2.6.2, $J$ is the annihilator of some $L(\lambda)$ and from Lemma $2.6.6.$ we have $J_\lambda\subset J$. 

If $\lambda(h)\notin\Z$ then by Theorem $2.6.3.$ we get $J=J_\lambda$. 

If $\lambda(h)\in\Z$ then Lemma $2.6.5.$, Theorem $2.6.1.$ and Example $2.4.6.$ imply that $J$ is either the annihilator of $L(\lambda)$ or $L(-\lambda)$. Moreover, we have $\lambda(h)\in\Z_{>0}$ or $\lambda(h)\in\Z_{\leq 0}$ . If $\lambda(h)\in Z_{\leq 0}$ then by Proposition $2.6.4.$ we get $L(\lambda)= M(\lambda)$, and hence $J=J_\lambda$. If $\lambda(h)\in\Z_{>0}$ then it is well known that $L(\lambda)$ is finite-dimensional, \cite{Bourbaki}. Hence $L(\lambda)\simeq S^k V$ for some $k\in\Z_{>0}$ and $J=Ann_{\UU(\g)}S^kV$. 
\end{proof}
\vspace{0.2cm}
\begin{myrmk}
Note that if $i<j$ then $e^{i+1}$ and $f^{i+1}$ annihilate $S^i V$ but they do not annihilate $S^jV$. Hence $Ann_{\UU(\g)}S^iV\neq Ann_{\UU(\g)}S^jV$ if $i\neq j$. 
\end{myrmk}
\vspace{0.2cm}
\begin{myprop}
Set $\h=\C h$. If $J=Ann_{\UU(\g)}S^kV$ for some $k\in\Z_{\geq 0}$ then $J\neq J_\lambda$ for all $\lambda\in\h^*$. 
\end{myprop} 
\begin{proof} 
Let $J=Ann_{\UU(\g)}S^kV$. Notice that Proposition $2.5.2.$ implies that, for all $\lambda$, $M(\lambda)$ is free as a $\C[f]$-module. But then $f^{k+1}$ cannot annihilate any $M(\lambda)$. Taking into account Proposition $2.5.9.$, the argument is complete. 
\end{proof}
Define an equivalence relation $\sim$ on $\h^*$ by $\lambda\sim -\lambda$. Then, for $[\lambda]\in\h^*/\sim$,  set $J_{[\lambda]}=J_\lambda=J_{-\lambda}$.\footnote[1]{It is trivial to see that this is an equivalence relation} Also set $I_k=Ann_{\UU(\g)}S^kV$. The following theorem completes the classification of primitive ideals of $\UU(sl(2,\C))$. 
\begin{mythm}
Let $\g=sl(2,\C)$ and $\h=\C h$. The list $\{I_k:k\in\Z_{\geq 0}\}\cup\{J_{[\lambda]}:[\lambda]\in\h^*/\sim\}$ is a complete classification of distinct primitive ideals of $\UU(\g)$. 
\end{mythm}
\begin{proof}
The statement follows from Proposition $2.6.6.$, Theorem $2.6.1.$, Remark $2.6.7.$ and Proposition $2.6.8.$
\end{proof}



\newpage
\section{Infinite-dimensional Lie algebras. Tensor modules. }

In this section we introduce the infinite-dimensional Lie algebras $gl(\infty, \C), sl(\infty, \C),$ $o(\infty, \C), sp(\infty, \C)$ and present a construction of the simple tensor modules. Relevant references here are \cite{PStyr} and \cite{PS}. 
 
\subsection{The infinite-dimensional Lie algebras : $gl(\infty, \C), sl(\infty, \C), o(\infty, \C), sp(\infty, \C)$}

Set $\JJ=\Z\setminus\{0\}$. Let $V$ and $W$ be countable-dimensional vector spaces and let $\langle\cdot,\cdot\rangle:W\times V\longrightarrow \C$ be a non-degenerate pairing. The Lie algebra $gl(\infty, \C)$ is defined as the tensor product $V\otimes W$ equipped with the following Lie bracket : 
$$[v_1\otimes w_1, v_2\otimes w_2] = \langle w_1, v_2\rangle v_1\otimes w_2 - \langle w_2, v_1\rangle v_2\otimes w_1.\mbox{ }\mbox{ }\mbox{ }\mbox{ }\mbox{ }\mbox{(3)}$$
A result by G. Mackey states that, given two countable-dimensional vector spaces $V$ and $W$ with a non-degenerate pairing, there always exists bases, $(v_i)_{i\in\JJ}$ of $V$ and $(w_i)_{i\in\JJ}$, dual to each other with respect to the pairing, i.e. $\langle w_i, v_j\rangle=\delta_{ij}$ where $\delta_{ij}$ is Kronecker's delta. Taking this into account, we can think of the Lie algebra $gl(\infty, \C)$ as the space of finitary infinite matrices via the mapping $\phi : v_j\otimes w_i\mapsto E_{i,j}$ for $i,j\in\JJ$, where the $E_{i,j}$ are the standard coordinate matrices and the bracket is the usual one. Here finitary means that the respective matrices have finitely many nonzero entries. Given this, we can define the Lie algebra $sl(\infty, \C)$ both as the kernel of the $\langle\cdot,\cdot\rangle$ pairing or as the subalgebra of finitary infinite traceless matrices. 

Set $\JJ_n$ to be the set consisting of the first $n$ elements of the sequence $1, -1, 2, -2,\ldots$ Define $V_n=span\{v_i\}_{i\in\JJ_n}$ and $W_n=span\{w_j\}_{j\in\JJ_n}$. Notice that the tensor product $V_n\otimes W_n$ equipped with the Lie bracket given by (3) is a subalgebra of $gl(\infty, \C)$ isomorphic to $gl(n, \C)$. To see this, observe that the restriction of the pairing $\langle\cdot,\cdot\rangle$ to $W_n\times V_n$ implies a restriction of $\phi$ to $V_n\otimes W_n$. In this way we get that $V_n\otimes W_n$ is the algebra spanned by $E_{i,j}$ with $i,j\in\JJ_n$. Similarly, the kernel of the restriction of the pairing $\langle\cdot,\cdot\rangle$ to $W_n\times V_n$ is a subalgebra of $sl(\infty, \C)$ isomorphic to $sl(n, \C)$. From this point on, when we talk about the algebras $gl(n, \C)$ and $sl(n, \C)$, we shall assume that they are embedded in $gl(\infty, \C)$ or $sl(\infty, \C)$ in the way described above.  

It is easy to see that $gl(n, \C)$ and $sl(n, \C)$ are subalgebras of $gl(n+1, \C)$ and $sl(n+1, \C)$ via the inclusion 
$$A\mapsto\left(\begin{array}{cc}A&0\\0&0\end{array}\right). $$
We immediately get $gl(\infty, \C)=\displaystyle\lim_{\longrightarrow}gl(n, \C)$ and $sl(\infty, \C)=\displaystyle\lim_{\longrightarrow}sl(n, \C)$, where the direct limit is taken using inclusions like above. The subalgebras $gl(n, \C)$ and $sl(n, \C)$ are called standard exhaustions of $gl(\infty, \C)$, or respectively $sl(\infty, \C)$. 

The Lie algebras $o(\infty, \C)$ and $sp(\infty, \C)$ can be defined in invariant terms in a similar manner by using a symmetric and respectively, an antisymmetric bilinear form on $V$. However, for simplicity reasons, we will just mention that $o(\infty, \C)$ is the subalgebra of $gl(\infty, \C)$ spanned by $(E_{i,j} - E_{-i, -j})_{i,j\in \JJ}$ and that $sp(\infty, \C)$ is the subalgebra spanned of $(sign(j)E_{i,j}-sign(i)E_{-j, -i})_{i,j\in \JJ}$. The subalgebras $o(n, \C)$ and $sp(n, \C)$ are then obtained by restricting the index set $\JJ$ to $\JJ_n$. 

One usually denotes $W$ by $V_*$ (and $W_n$ by $V_{n}^*$), also pointing in this way to its role as a "continuous" dual to $V$ (see \cite{PS}). For $p, q\geq 0$ we endow the space of mixed tensors $V^{\otimes(p, q)}=V^{\otimes p}\otimes V_*^{\otimes q}$ with the $gl(\infty, \C)$-module structure  
$$(u\otimes u^*)\cdot(v_1\otimes\ldots\otimes v_p \otimes v_1^*\otimes\ldots\otimes v_q^*) = $$
$$= \displaystyle\sum_{i=1}^{p}\langle u^*, v_i\rangle v_1\otimes\ldots\otimes v_{i-1}\otimes u\otimes v_{i+1}\otimes\ldots\otimes v_p\otimes v_1^*\otimes\ldots\otimes v_q^* - $$
$$ - \displaystyle\sum_{j=1}^{q}\langle v_j^*, u\rangle v_1\otimes\ldots\otimes v_p\otimes v_1^*\otimes\ldots\otimes v_{j-1}^*\otimes u^*\otimes v_{j+1}\otimes\ldots\otimes v_q^*. $$ 



\vspace{0.2cm}

\begin{myrmk}
If $\g$ is one of the subalgebras $sl(\infty, \C), o(\infty, \C), sp(\infty, \C)$ of $gl(\infty, \C)$ then $V$ (respectively $V_*$) is a module over $\g$ and it is termed the natural $\g$-module (respectively conatural $\g$-module). The associated representation is termed the natural representation (respectively conatural representation) of $\g$. Similarly, $V_n$ and $V_{n}^*$ are modules over the algebras $gl(n, \C), sl(n, \C), o(n, \C)$ and $sp(n, \C)$. In the case when $\g\simeq o(\infty, \C), sp(\infty, \C)$ one easily obtains $V\simeq V_*$. 
\end{myrmk}

The vectors in $V_n$, respectively $V_{n}^*$, can be viewed in a natural way as column, respectively row, vectors with $n$ entries, by mapping $v_i\mapsto e_i$ and $w_j\mapsto e_j^T\mbox{ }i,j\in\JJ_n$, where the $e_i$ are the standard coordinate vectors. It is clear that $V_n$, respectively $V_{n}^*$, is a $gl(n, \C)$-submodule of $V_{n+1}$, respectively $V_{n+1}^*$, using the inclusion 
$$u\mapsto\left(\begin{array}{c}u\\0\end{array}\right). $$ 
Of course, that this also holds in the cases of $sl(n, \C), o(n, \C)$ and $sp(n, \C)$, and in all cases $V$, respectively $V_*$, is isomorphic to the direct limit $\displaystyle\lim_{\longrightarrow}V_n$, respectively $\displaystyle\lim_{\longrightarrow}V_n^*$.  

The natural and conatural representation of $\g$ can also be characterized in invariant terms. For instance, one can show that up to isomorphism, $V$ is the only simple module of $\g$ for which there exists an exhaustion of $\g$ by simple Lie algebras $\g_1\subset\g_2\subset\ldots$ such that $V$ restricted to $\g_n$ is isomorphic to the natural representation of $\g_n$ plus a trivial module for all $n$. A similar statement holds for $V_*$.  
\vspace{0.2cm}


\subsection{Tensor Modules}

As we will see, simple tensor modules are highest weight modules with respect to a certain choice of the Borel subalgebra. Also, we will point out that they are not, in general, highest weight modules with respect to "the standard" choice of Borel subalgebra. In this section, we present the construction of simple tensor modules for $gl(\infty, \C)$ and $sl(\infty, \C)$. For $o(\infty, \C)$ and $sp(\infty, \C)$, the constructions are similar and they can be found in great detail in \cite{PStyr}. 

Consider the following decomposition of $gl(\infty, \C)$ 
$$gl(\infty, \C)=\h_{gl}\oplus(\bigoplus_{\alpha\in\Delta}\C X_\alpha), $$
where 
$$\h_{gl}=\displaystyle\bigoplus _{i\in\Z\setminus\{0\}}\C E_{i, i} \mbox{ }\mbox{ }\mbox{ }\mbox{ }\mbox{ }\Delta = \{ \varepsilon_i - \varepsilon_j : i,j \in \Z \setminus \{ 0 \}, i\neq j\}. $$
Here $\varepsilon_i\in\h_{gl}^*$ with $\varepsilon_i(E_{j, j})=\delta_{ij}$ and $X_{\varepsilon_i-\varepsilon_j}=E_{i, j}$. We have $[h,X_\alpha]=\alpha(h)X_\alpha^{gl}$ for all $\alpha\in\Delta$ and $h\in\h_{gl}$. 

In the finite-dimensional case, the choice of a Borel subalgebra of $gl(n, \C)$ corresponds to an ordering of the set $\{1, 2,\ldots,n\}$, or equivalently, to a splitting of the set of roots $\Delta_{gl(n, \C)}$ into positive and negative roots. In the infinite-dimensional case, the choice of a Borel subalgebra corresponds to linear orders on $\JJ$.  Taking this into account, define the set of positive roots as follows 
$$\Delta^+=\{\varepsilon_i-\varepsilon_j : 0<i<j\}\bigcup\{\varepsilon_i-\varepsilon_j : i<j<0\}\bigcup\{\varepsilon_i-\varepsilon_j : j<0<i\}. $$
This choice of the positive roots corresponds to the order $1>2>\ldots>-2>-1$. Set $\n_+=\displaystyle\oplus_{i,j} \C E_{i,j}$, where the sum is taken over all pair of indexes $(i,j)$ with the property that $\varepsilon_i-\varepsilon_j\in\Delta^+$. The Borel subalgebra is then $\b_{gl}^+=\h_{gl}\oplus\n_+$. By $\b_{gl}^-$ we denote the opposite Borel subalgebra. When referring to a Cartan subalgebra or Borel subalgebra of one of the Lie algebras $gl(\infty, \C), sl(\infty, \C), o(\infty, \C)$ or $sp(\infty, \C)$, we will omit the abbreviation of the Lie algebra from which it comes and just write $\h, \b_+$ or $\b_-$ where there is no confusion. 

Now define a partition $\lambda$ as a finite decreasing set of positive integers : $\lambda = (\lambda_1 \geq \lambda_2 \geq\ldots\geq \lambda_k)$ (the empty partition is denoted by $0$) and let $\lambda^T$ be the dual partition given by $\lambda^T_i=\#\{j : \lambda_j\geq i\}$. These partitions have a combinatorial interpretation via the so called Young diagrams, \cite{GWl}. The Young diagram of the dual partition is obtained by transposing the Young diagram of $\lambda$. Define $|\lambda|=\displaystyle\sum_{i=1}^{k} \lambda_i$.

For a partition $\lambda$ denote by $c_\lambda$ the Young projector corresponding to the Young diagram of shape $\lambda$, \cite{GW}.  For $|\lambda|=d$, $c_\lambda$ acts naturally on $V^{\otimes d}$. Let 
$$\SSS_\lambda V=im(c_\lambda : V^{\otimes d}\longrightarrow V^{\otimes d}). $$ 

Following the theory of irreducible representations of the symmetric group $\SS_d$, denote  
$$H_\lambda = \C[\SS_d]c_\lambda. $$

Now we can proceed to the construction in \cite{PStyr}. For $i\in\{1,\ldots ,p\}$ and $j\in\{1,\ldots,q\}$, consider the contraction $\Phi_{i,j}:V^{\otimes(p,q)}\longrightarrow V^{\otimes(p-1,q-1)}$ given by  
$$\Phi_{i,j}(v_1\otimes\ldots\otimes v_p\otimes v_1^*\otimes\ldots\otimes v_q^*)=\langle v_j^*,v_i\rangle v_1\otimes\ldots\otimes \widehat{v_i}\otimes\ldots\otimes v_p\otimes v_1^*\otimes\ldots\otimes \widehat{v_j^*}\otimes\ldots\otimes v_q^*. $$
The hat in the right hand side denotes that the respective vector is omitted from the tensor product. Let $V^{\{p,q\}}=\displaystyle\bigcap_{i,j} \ker \Phi_{i,j}$. By convention $V^{\{p,0\}}=V^{\otimes p}$ and $V^{\{0,q\}}=V_*^{\otimes q}$. Lastly, for partitions $\lambda$ and $\mu$ with $|\lambda|=p$ and $|\mu|=q$, we set 
$$V_{\lambda\mu}=V^{\{p,q\}}\bigcap (\SSS_\lambda V\otimes \SSS_\mu V_*). $$ 

When we refer to pairs of partitions we use the notation $(\lambda,\mu)$. However, for simplicity reasons, we use the notation $V_{\lambda\mu}$ instead of $V_{(\lambda,\mu)}$. The following theorem is one of the results in \cite{PStyr} concerning tensor modules. 

\begin{mythm}
For any $p,q$ there is an isomorphism of $(gl(\infty, \C), \SS_p\times\SS_q)$-modules : 
$$V^{\{p,q\}}\simeq \displaystyle\bigoplus_{|\lambda|=p, |\mu|=q}V_{\lambda\mu}\otimes (H_\lambda\otimes H_\mu). $$
For any partitions $\lambda, \mu$, the $gl(\infty, \C)$-module $V_{\lambda\mu}$ is a simple highest weight module with highest weight $\chi=\displaystyle\sum_{i\in\\Z_{>0}} \lambda_i\varepsilon_i-\displaystyle\sum_{i\in\Z_{>0}}\mu_i\varepsilon_{-i}$. Furthermore, $V_{\lambda\mu}$ is simple when regarded by restriction as an $sl(\infty, \C)$-module. 
\end{mythm}     
\vspace{0.2cm}
\begin{myrmk}
There exist an alternative way to introduce the simple tensor modules. In \cite{PS} it is proven that the simple subquotients of all the modules $V^{\otimes(p,q)}$ are exactly the simple tensor modules. As we already know, for fixed $p, q\in\N$ and partitions $\lambda,\mu$ with $|\lambda|=p$ and $|\mu|=q$, $V_{\lambda\mu}$ is a submodule of $V^{\otimes(p,q)}$. 
\end{myrmk}
\vspace{0.2cm}
The standard Borel subalgebra for $gl(\infty, \C)$ and $sl(\infty, \C)$ corresponds to the order $1>-1>2>-2>\ldots$. Denote this subalgebra by $\b_+^{st}$. It is to see that the simple tensor modules $V_{\lambda 0}$ and $V_{0 \mu}$ are also highest weight modules with respect to $\b_+^{st}$, and moreover the highest weight spaces coincide for the two choices, $\b_+$ and $\b_+^{st}$.

Now suppose $V_{\lambda\mu}$ is a highest weight module with respect to $\b_+^{st}$, $|\lambda|=p$ and $|\mu|=q$. Let $\chi'$ be a highest weight vector with respect to $\b_+^{st}$. Note that it is a weight vector with respect to $\h_{gl}$. Let $\{x_i\}_{i\in\JJ}$ and $\{y_i\}_{i\in \JJ}$ be the dual bases of $V$ and $V_*$ used in section 3.1. to construct $gl(\infty, \C)$. 

Set $\chi'=\sum_{i\in I} u_i\otimes v_i$ where $u_i\in V^{\otimes p}$ and $v_i\in\otimes V_*^{\otimes q}$ are pure tensors of type $c_i x_{i_1}\otimes\ldots\otimes x_{i_p}$ respectively $d_jy_{j_1}\otimes\ldots\otimes y_{j_q}$ and $I$ is some finite index set, $c_i,d_i\in \C$. Using that $\n_+^{st}\chi'=0$ and that $\chi'$ is a weight vector, we see that $ u_i\in\C x_1\otimes\ldots x_1$ and $v_i \in \C y_{-1}\otimes\ldots y_{-1}$, where $x_1$ is tensored $p$ times and $y_{-1}$, $q$ times. But then this implies that $x_1\otimes\ldots\otimes x_1\otimes y_{-1}\otimes\ldots\otimes y_{-1}$ is a highest weight vector with respect to $\b_+^{st}$. This cannot happen unless the $x_1$ or $y_{-1}$ is missing from the tensor product, or equivalently if $pq=0$. Thus $V_{\lambda\mu}$ is highest weight with respect to $\b_+^{st}$ if and only if $pq=0$. 
\newpage
\section{Results on annihilators of simple tensor modules}
\vspace{0.1cm}
Before we proceed to proving the main theorems of this thesis, we need some preliminary results. We split the preliminaries into two sections : one section concerning finite-dimensional simple modules and finite-dimensional semisimple Lie algebras, and one section concerning integrable modules and infinite-dimensional locally semisimple Lie algebras. In each section, the statements do not follow any particular order and for those results that are well known, we will skip the proof. 

\subsection{Preliminaries concerning finite-dimensional modules over semisimple Lie algebras}
\vspace{0.1cm}
In this section $\g$ denotes a finite-dimensional semisimple Lie algebra and $\h$ a Cartan subalgebra of $\g$. 
\begin{mydef}\

Let $\Delta$ be the set of roots of $\g$ and $\Delta_+$ the set of positive roots. For each root $\alpha\in\Delta$, denote by $H_\alpha\in\h$ the corresponding coroot. 
\begin{enumerate}
\item A weight $\lambda\in\h^*$ is called integral if $\lambda(H_\alpha)\in\Z$ for all $\alpha\in\Delta_+$. 
\item A weight $\lambda\in\h^*$ is called dominant if $\lambda(H_\alpha)\geq 0$ for all $\alpha\in\Delta_+$. 
\end{enumerate}
\end{mydef}
\vspace{0.1cm}
\begin{myprop}\cite{Dixmier1}

Set $\rho\in\h^*$ to be the half-sum of the positive roots. If $L(\lambda)$ is a finite-dimensional simple $\g$-module then $\lambda-\rho$ is integral and dominant.  
\end{myprop}
\vspace{0.1cm}
\begin{myprop}\cite{Humphreys}, \cite{FH}

Let $W$ be the Weyl group of $\g$. For $\lambda\in\h^*$ and $w\in W$ define $w\cdot\lambda=w(\lambda+\rho)-\rho$. Then for any integral weight $\lambda$, there exist at most one $w\in W$ such that $w\cdot\lambda$ is a dominant weight.  
\end{myprop}

\vspace{0.1cm}

Now let $M$ be a simple $\g$-module. We know that $(Ann_{\UU(\g)}M)\cap Z(\g)=ker\mbox{ }\chi_\lambda$ for some $\lambda\in\h^*$. This is a maximal ideal in $Z(\g)$ and $Z(\g)$ is a polynomial algebra in $r$ indeterminates where $r=rank \mbox{ }\g$. Hilbert's Nullstellensatz gives us that the associated variety of this maximal ideal is a point in $\C^r$. Denote this point by $Q(M)$. If $M=L(\mu)$ for some $\mu\in\h^*$ then we also denote $Q(M)$ by $Q(\mu)$. 

\vspace{0.1cm}
The following claim is an immediate consequence of Proposition $4.1.3.$

\begin{myprop}
If $M$ and $N$ are finite-dimensional simple $\g$-modules then $Q(M)=Q(N)$ if and only if $M\simeq N$. 
\end{myprop}
\begin{proof}
Set $M=L(\lambda+\rho)$ and $N=L(\mu+\rho)$. Then $Q(M)=Q(N)$ is equivalent with $ker\mbox{ }\chi_{\lambda+\rho}=ker\mbox{ }\chi_{\mu+\rho}$ which by part $2$ of Theorem $2.6.1.$ is equivalent with $\mu+\rho\in W(\lambda+\rho)$ or, $\mu=w\cdot\lambda$ for some $w\in W$. Apply Proposition $4.1.2.$  and Proposition $4.1.3.$ for $\lambda$ and $\mu$. Conclusion follows easily. 
\end{proof}

\vspace{0.1cm}
\begin{mylma}\cite{Bourbaki}
 
Let $x$ be a semisimple element of $\g$. Then there exist a Cartan subalgebra $\h_0\subset\g$ such that $x\in\h_0$. 

\end{mylma}
\vspace{0.1cm}
We conclude this section with a theorem that was proven independently S. Fernando and V. Kac.  in the 1980's. 
\begin{mythm}
Let $\g$ be a finite-dimensional Lie algebra. Let $M$ be a $\g$-module and set $\g[M]=\{g\in\g : \dim span\{m, gm, g^2m,\ldots\}<\infty, \forall m\in M\}$. Then $g[M]$ is a subalgebra of $\g$.  
\end{mythm}
\begin{proof}
We will present just a sketch of Kac's proof from \cite{Kac}. A first easy observation is that if $x\in\g[M]$ then for any $t\in\C$ we have $tx\in\g[M]$. Next step is to prove that if $z_1, z_2\in\g[M]$ then $z_1+z_2\in\g[M]$ thus showing, combined with the previous observation, that $\g[M]$ is a subspace of $\g$. This is done using the Poincare-Birkhoff-Witt theorem. 

Now pick $x,y\in \g$ and $t\in\C$. One can prove the following identity $e^{tad_x}y\cdot m=e^{tx}ye^{-tx}\cdot m, \forall m\in M$. Given that $(e^{tx}ye^{-tx})^n=e^{tx}y^ne^{-tx}$, we obtain that if $x, y\in\g[M]$ then $e^{tad_x}y\in\g[M]$. Define the element
$$F(t)=\displaystyle\frac{1}{t}(e^{tad_x}y-y)$$ 
and set $A=span(\{y\}\cup\displaystyle\bigcup_{t\in\C}\{F(t)\})$. $A$ is a subspace of $\g$ so it is finite-dimensional. This means it is spanned by finitely many elements of type $tF(t)+y=e^{tad_x}y\in\g[M]$. In particular, $A\subset \g[M]$. Hence $\displaystyle\lim_{t\rightarrow 0}F(t)\in A\subset\g[M]$ and so $[x,y]\in\g[M]$. This concludes that $\g[M]$ is a subalgebra of $\g$.  
\end{proof}

\begin{myrmk} 
In the proof of Theorem $4.1.6.$, the only place where the finite-dimensionality of $\g$ is used is when arguing that $A$ is finite-dimensional as a subspace of $\g$. However, notice that for $\g$ infinite-dimensional and locally finite, we can include $x, y$ into a finite-dimensional subalgebra $\g_0$ of $\g$ and then include $A$ in $\g_0$. The rest of the argument follows through. Thus the statement of Theorem $4.1.6.$ holds in a greater generality and in particular it holds for the Lie algebras we are interested in : $gl(\infty, \C), sl(\infty, \C), o(\infty, \C)$ and $sp(\infty, C)$. 
\end{myrmk}
\subsection{Preliminaries on integrable modules and infinite-dimensional locally\newline simple Lie algebras}
\vspace{0.1cm}
In this section $\g$ denotes an infinite-dimensional locally simple Lie algebra unless otherwise stated.
\begin{mylma}
Let $M, N$ be integrable $\g$-modules such that for any finite-dimensional semisimple subalgebra $\g_0$, the number of non-isomorphic simple constituents of $M$ and $N$ restricted to $\g_0$ is finite. Denote these sets by $SC(M, \g_0)$ and $SC(N, \g_0)$. Suppose $Ann_{\UU(\g)}N\subseteq Ann_{\UU(\g)}M$. Then for any finite-dimensional semisimple Lie subalgebra $\g_0\subset\g$, we have $SC(M, \g_0)\subseteq SC(N, \g_0)$. If $Ann_{\UU(\g)}N=Ann_{\UU(\g)}M$ then $SC(N,\g_0)=SC(M,\g_0)$ for any finite-dimensional semisimple subalgebra $\g_0\subset\g$.  
\end{mylma}
\begin{proof}
Pick a finite-dimensional semisimple Lie subalgebra $\g_0$. Let $m=\# SC(M, \g_0)$ and $n=\# SC(N, \g_0)$. Set 
$$M=\displaystyle\bigoplus_{i=1}^mHom_{\g_0}(X_i,M)\otimes X_i, $$
$$N=\displaystyle\bigoplus_{i=1}^nHom_{\g_0}(Y_i,N)\otimes Y_i, $$
where $X_i$ and $Y_i$ are respectively the elements of $SC(M, \g_0)$ and $SC(N, \g_0)$, and $Hom_{\g_0}(X_i, M)$, $Hom_{\g_0}(Y_i, M)$ are trivial $\g_0$-modules. Set $X_i=L(\lambda_i), i\in\{1,\ldots,,m\}$ and $Y_j=L(\mu_j), j\in\{1,\ldots,n\}$. We have $(Ann_{\UU(\g)}N)\cap Z(\g_0)\subseteq(Ann_{\UU(\g)}M)\cap Z(\g_0)$. Hence 
$$\displaystyle\bigcap_{i=1}^n ker\mbox{ }\chi_{\mu_i}\subseteq\displaystyle\bigcap_{j=1}^m ker\mbox{ }\chi_{\lambda_j}. $$  
This inclusion of ideals in $Z(\g_0)$ implies a reversed inclusion of their associated varieties. Hence 
$$\{Q(X_1),\ldots,Q(X_m)\}\subseteq\{Q(Y_1),\ldots,Q(Y_n)\}. $$
But then Proposition $4.1.4.$ completes the argument. 

The last part of the statement follows using the same line of reasoning but replacing inclusions with equalities. 
\end{proof}

\vspace{0.1cm}
\begin{myprop}
Let $M, N$ be $\g$-modules such that $Ann_{\UU(\g)}N\subseteq Ann_{\UU(\g)}M$. If $N$ is integrable and if, for any finite-dimensional semisimple subalgebra $\g_0\subset\g$, $SC(N, \g_0)$ is a finite set, then $M$ is an integrable $\g$-module and for any finite-dimensional semisimple subalgebra $\g_0\subset \g$, $SC(M, \g_0)\subseteq SC(N, \g_0)$. If $Ann_{\UU(\g)}N=Ann_{\UU(\g)}M$ then $SC(N,\g_0)=SC(M,\g_0)$ for any finite-dimensional semisimple subalgebra $\g_0\subset\g$.  
\end{myprop}
\begin{proof} We first show that $M$ is integrable. Let $\g_0$ be an arbitrary finite-dimensional semisimple subalgebra of $\g$. Pick $h\in\g_0$ semisimple. By Lemma $4.1.5.$, $h$ lies in some Cartan subalgebra $\h_0$ of $\g_0$. Since $N$ is integrable, its restriction to $\g_0$ is a sum of finite-dimensional simple modules which are also weight modules with respect to $\h_0$. Let $\Lambda(\g_0, \h_0)$ be the union of the supports of the elements of $SC(N, \g_0)$. 

Since $SC(N, \g_0)$ is finite, we obtain that $\Lambda(\g_0, \h_0)$ is finite. Set $z_0=\displaystyle\prod_{\lambda\in\Lambda(\g_0, \h_0)}(h-\lambda(h))\in\UU(\g_0)$. Then it is easy to check that $z_0\in Ann_{\UU(\g_0)}N$. Since $Ann_{\UU(\g)}N\subseteq Ann_{\UU(\g)}M$, we get $Ann_{\UU(\g_0)}N\subseteq Ann_{\UU(\g_0)}M$. Hence $z_0\in Ann_{\UU(\g_0)}M$. Taking into account that $z_0$ is a polynomial in $h$ we get $h\in\g_0[M]$. This implies that the Fernando-Katc subalgebra $\g_0[M]$ contains all semisimple elements of $\g_0$. As $\g_0$ is generated by its semisimple elements, so we get, using Theorem $4.1.6.$, that $\g_0[M]=\g_0$. 

We have $\g_0=\g_0[M]\subset\g[M]$. Letting $\g_0$ be any finite-dimensional subalgebra of $\g$ from our fixed standard exhaustion, we see that $\g[M]=\g$. Thus $M$ is integrable. 

Next step is to prove that $SC(M, \g_0)$ is finite. We have 
$$(Ann_{\UU(\g_0)}N)\cap Z(\g_0)\subseteq(Ann_{\UU(\g_0)}M)\cap Z(\g_0). \mbox{ }\mbox{ }\mbox{ }\mbox{ }\mbox{ }\mbox{ }\mbox{ (3) }$$ 
Let $V_1\subset\C^{rank\mbox{ }\g_0}$ be the associated variety of the ideal from the left hand side and $V_2\subset\C^{rank\mbox{ }\g_0}$ be the associated variety of the ideal from the right hand side of (3). We have $V_2\subseteq V_1$. As $SC(N, \g_0)$ is finite, we get that $V_1$ is finite, and hence $V_2$ is finite. But this can happen if only if $SC(M, \g_0)$ is finite. Here we have used that if $(X_i)_{i\in I}\subset SC(M, \g_0)$ for some index set $I$, then $\{Q(X_i) : i\in I\}\subset V_1$ ; the inclusion becomes equality if $I$ is finite.  Thus $SC(M, \g_0)$ is finite. 

The proof is complete by applying Lemma $4.2.1.$ 
\end{proof}

Now we turn our attention to simple tensor modules. For this purpose, we introduce a relation on pairs of partitions $(\lambda,\mu)$ which will turn out to be a partial order on the set of simple tensor modules. In the rest of the section, $\g$ will denote one of the Lie algebras $gl(\infty, \C), sl(\infty, \C), o(\infty, \C)$ or $sp(\infty, \C)$. 

\begin{mydef}\

\begin{enumerate}
\item For a pair of partitions $(\lambda,\mu)=(\lambda_1\geq\lambda_2\geq\ldots\geq\lambda_p,-\mu_q\geq -\mu_{q-1}\geq\ldots\geq -\mu_1)$ define 
$$GT((\lambda,\mu)):=\{(\lambda_1'\geq\lambda_2'\geq\ldots\geq\lambda_p',-\mu_q'\geq -\mu_{q-1}'\geq\ldots\geq -\mu_1') : $$
$$ \lambda_{i+1}\leq\lambda_i'\leq\lambda_i, i\in\{1,\ldots,p\}\mbox{ }\lambda_{p+1}=0\mbox{ }\mbox{ }\mbox{ and }\mbox{ }\mbox{ }\mu_{i+1}\leq\mu_i'\leq\mu_i, i\in\{1,\ldots,q\}\mbox{ }\mu_{q+1}=0\}. $$ 
\item For a partition $\lambda=(\lambda_1\geq\lambda_2\geq\ldots\geq\lambda_n)$ define
$$GT(\lambda):=\{(\mu_1,\mu_2,\ldots\,\mu_{p-1}) : \lambda_{i+1}\leq\mu_i\leq\lambda_i, i\in\{1,\ldots,p-1\}\}. $$
\item For a set $S$ of pairs as in $1.$ or partitions as in $2. $, define $GT(S):=\displaystyle\bigcup_{x\in S}GT(x)$. 
\item For two pair of partitions $(\lambda,\mu)$ and $(\lambda',\mu')$, we say that $(\lambda',\mu')\succeq(\lambda,\mu)$  or $(\lambda,\mu)\preceq (\lambda',\mu')$ if there exist $i\in\Z_{\geq 0}$ such that $(\lambda,\mu)\in GT^{\circ i}(\{(\lambda',\mu')\})$, where $GT^{\circ i}(x)=\underbrace{GT(GT(\ldots GT}_{i \mbox{ times}}(x))\ldots)$.  
\end{enumerate}
\end{mydef}
\begin{myprop}
The relation $\succeq$ is a partial order. 
\end{myprop}
\begin{proof} Reflexivity follows from the fact that $(\lambda,\mu)\in T^{\circ 0}(\{(\lambda,\mu)\})$. 

Antisymmetry is a consequence of the fact that if $(\lambda,\mu)\succeq (\lambda',\mu')$ and $(\lambda',\mu')\succeq (\lambda,\mu)$, we get inequalities of type $\lambda_i'\leq\lambda_i, \mu_j'\leq\mu_j, i\in\{1,\ldots,p\}, j\in\{1,\ldots,q\}$ and $\lambda_i\leq \lambda_i', \mu_i\leq\mu_i'$ which obviously implies $(\lambda,\mu)=(\lambda',\mu')$.
 
Lastly, if $x, y, z$ are pairs of partitions such that $x\succeq y$ and $y\succeq z$ then $z\in T^{\circ j}(\{y\})$ and $y\in T^{\circ i}(\{x\})$ for some $i,j\in \Z_{\geq 0}$. But then $z\in T^{\circ (i+j)}(\{x\})$ and hence $x\succeq z$. Thus $\succeq$ is transitive. 
\end{proof}

We denote by $\g_n$ one of the Lie algebras $gl(n, \C), sl(n, \C), o(n, \C)$ or $sp(n, \C)$\footnote[1]{If for instance $\g\simeq gl(\infty, \C)$ then all the $\g_n$ are the subalgebras from the standard exhaustion of $gl(\infty, \C)$}. For each $\g_n$ let $\h_n=\h\cap\g_n, \b_n^+=\g_n\cap\b_+$ and $\b_n^-=\g_n\cap\b_-$, where $\h$, $\b_+$ are the Cartan subalgebra and Borel subalgebra defined in Section $3.2.$ and $\b_-$ is the opposite Borel subalgebra. The subalgebras $\h_n, \b_n^+$ and $\b_n^-$ are Cartan and Borel subalgebras of $\g_n$, see for instance \cite{PStyr}. 

For a partition $\lambda=(\lambda_1, \lambda_2, \ldots, \lambda_n)$ we denote\footnote[2]{For simplicity reasons, we deviate from our standard notation $L(\lambda)$} by $F_n^\lambda$ the simple highest weight $\g_n$-module of highest weight $\lambda$ and by $F_{n+1\downarrow n}^\lambda$, the restriction of $F_{n+1}^\lambda$ to $\g_n$. 

The following statement motivates the partial order defined above. 

\begin{mythm}(Gelfandt-Tsetlin, see \cite{HTW})

Consider the inclusion $\g_n\subset\g_{n+1}$. For a partition $\lambda$ we have  
$$F_{n+1\downarrow n}^\lambda\simeq\displaystyle\bigoplus_{\mu\in GT(\lambda)}F_n^\mu. \mbox{ }\mbox{ }\mbox{ }\mbox{ }\mbox{ }\mbox{(4)}$$ 

\end{mythm}
\vspace{0.1cm}

\begin{mylma}

If $\xi$ is a highest weight vector  with respect to $\b_{n+1}^+$ of $F_{n+1}^\lambda$ then $\xi$ is also a highest weight vector  with respect to $\b_{n}^+$ of the simple module $F_n^{\lambda\downarrow n}$ in the restriction of $F_{n+1\downarrow n}^\lambda$, where $\lambda\downarrow n$ is the restriction of $\lambda$ to $\h_n$. 
\end{mylma}
\begin{proof} The statement is obvious. 
\end{proof}


\begin{mylma}
If there exists $x\in\g$ such that $x^2\in Ann_{\UU(\g)} V\cap Ann_{\UU(\g)} V_*$, then $x^{p+q+1}\in Ann_{\UU(\g)} V^{\otimes (p,q)}$. 
\end{mylma}
\begin{proof}
First let's show that, for $k\leq p+q$, 
$$x^kv_1\otimes v_2\otimes\ldots\otimes v_p\otimes v_{ *(p+1)}\otimes v_{ *(p+2)}\otimes\ldots\otimes v_{ *(p+q)}=$$
$$=\displaystyle\sum_{1\leq i_1<i_2<\ldots<i_k\leq p+q} v_1\otimes v_2\otimes\ldots xv_{i_r}\otimes\ldots\otimes v_p\otimes v_{ *(p+1)}\otimes v_{ *(p+2)}\otimes\ldots\otimes xv_{ *i_s}\otimes\ldots\otimes v_{ *(p+q)}, $$
where $v_i\in V$ and $v_{ *j}\in V_*$ and the sum is taken over all $k$-tuples $1\leq i_1<\ldots<i_k\leq p+q$ such that the factor $v_{i_j}$ (or $v_{*i_j}$) appears with $g$ in front. This follows by easy induction : base case is just the definition of the $\g$-action on $V^{\otimes (p,q)}$ ; the induction step follows by using the condition of $x^2$ being in the intersection of the annihilators of $V$ and $V_*$. 

We now have 
$$x^{p+q}v_1\otimes v_2\otimes\ldots\otimes v_p\otimes v_{ *(p+1)}\otimes v_{ *(p+2)}\otimes\ldots\otimes v_{ *(p+q)}=$$
$$=xv_1\otimes xv_2\otimes\ldots\otimes xv_p\otimes xv_{ *(p+1)}\otimes xv_{ *(p+2)}\otimes\ldots\otimes xv_{ *(p+q)}. $$
Hence  
$$x^{p+q+1}v_1\otimes v_2\otimes\ldots\otimes v_p\otimes v_{ *(p+1)}\otimes v_{ *(p+2)}\otimes\ldots\otimes v_{ *(p+q)}=$$
$$=\displaystyle\sum_{1\leq i\leq p}xv_1\otimes xv_2\otimes \ldots \otimes x^2v_i\otimes\ldots\otimes x.v_p\otimes gv_{ *(p+1)}\otimes xv_{ *(p+2)}\otimes\ldots\otimes xv_{ *(p+q)}+$$
$$+\displaystyle\sum_{p+1\leq i\leq p+q}xv_1\otimes xv_2\otimes\ldots\otimes xv_p\otimes xv_{ *(p+1)}\otimes xv_{ *(p+2)}\otimes\ldots\otimes x^2v_{*(i)}\otimes\ldots\otimes xv_{ *(p+q)}=$$ 
$$=\displaystyle\sum_{1\leq i\leq p} 0+\displaystyle\sum_{p+1\leq i\leq p+q} 0=0. $$
To complete the argument it suffices to observe that the elements of $V^{\otimes (p,q)}$ are just linear combinations of elements of the form $v_1\otimes v_2\otimes\ldots\otimes v_p\otimes v_{ *(p+1)}\otimes v_{ *(p+2)}\otimes\ldots\otimes v_{ *(p+q)}$. 
\end{proof}
\vspace{0.1cm}
\begin{mylma}
If there exist $x\in\g$ such that $x^k\in Ann_{\UU(\g)} V\cap Ann_{\UU(\g)} V_*$ for some $k\geq 2$, then $x^{(k-1)(p+q)+1} \in Ann_{\UU(\g)} V^{\otimes (p,q)}$. 
\end{mylma}
\begin{proof}
We prove the statement by induction. The base case follows from Lemma $4.2.7.$ Assume the claim to be true for all $1\leq i\leq k$ and let us prove it for $i=k+1$. After some easy computations and using the fact that $x^{(k-1)(p+q)+1}\in Ann_{\UU(\g)}V^{\otimes (p, q)}$, we get  
$$x^{k(p+q)}v_1\otimes v_2\otimes\ldots\otimes v_p\otimes v_{ *(p+1)}\otimes v_{ *(p+2)}\otimes\ldots\otimes v_{ *(p+q)}=$$
$$=x^kv_1\otimes x^kv_2\otimes\ldots\otimes x^kv_p\otimes x^kv_{ *(p+1)}\otimes x^kv_{ *(p+2)}\otimes\ldots\otimes x^kv_{ *(p+q)}. $$
Acting one more time with $g$ annihilates all summands and thus the conclusion follows. 
\end{proof}

\subsection{Non-triviality of annihilators of simple tensor modules}

In this section $\g\simeq gl(\infty, \C), sl(\infty, \C), o(\infty, \C)$ or $sp(\infty, \C)$. 
We will prove the non-triviality of the annihilators of simple tensor modules. We will give two separate arguments. The reason why we do so is because we want to exhibit two types of elements that appear in the annihilator : elements that are "locally central" and elements that are "non-locally central". Here by locally central elements we mean elements that appear in the center of the enveloping algebra of some finite-dimensional semisimple subalgebra of $\g$. 

\vspace{0.2cm}

\vspace{0.1cm}
\begin{myprop}

For any partitions $\lambda$ and $\mu$ with $|\lambda|=p, |\mu|=q$ we have $Ann_{\UU(\g)}V_{\lambda\mu}\neq \{0\}$. 
\end{myprop}
\begin{proof}\

\begin{enumerate}

\item Fix $p,q\in\N$. Note that from the construction in Section $3.2.$ it follows that $V_{\lambda\mu}$ is a submodule of $V^{\otimes (p,q)}$. This implies that $Ann_{\UU(\g)}V^{\otimes (p,q)}\subset Ann_{\UU(\g)}V_{\lambda\mu}$ so it suffices to show that $V^{\otimes (p,q)}$ has a non-trivial annihilator. It is then enough to find an element as in Lemma $4.2.7.$ or Lemma $4.2.8.$ The idea is that all the Lie algebras considered above are matrix algebras and because their action on the natural and conatural modules is nothing more than matrix multiplication (up to sign for the conatural module) on column and row vectors, it is enough to find a matrix $x\neq 0$ from one of the subalgebras from the standard exhaustions such that $x^k=0$ as a matrix for some $k\geq 2$. Such a $x$ would suffice since $x^{(k-1)(p+q)+1}\in\UU(\g)$ is non zero as $\UU(\g)$ is integral, so we can apply Lemma $4.2.8.$  
In the case $\g\simeq sl(\infty, \C), gl(\infty, \C)$ choose $x=E_{1,2}$ and $k=2$ ; in the case $\g\simeq o(\infty, \C)$ choose $x=E_{1,2}-E_{-1,-2}$ and $k=2$ ; lastly, in the case $\g\simeq sp(\infty, \C)$ choose $x=E_{1,2}-E_{-2,-1}$ and $k=2$. 

\item Another approach is by restricting $V_{\lambda\mu}$ to any finite-dimensional semisimple subalgebra $\g_0$. In \cite{PS} it is proven that $SC(V_{\lambda\mu},\g_0)$ is finite. Then pick any $z\in Z(\g_0),z\neq 0$ (for instance take the Casimir operator). On each simple $\g_0$-module from the restriction, $z$ acts by a scalar according to Proposition $2.3.8.$ Let $c_1,c_2,\ldots\,c_s$ be the scalars for each of the $s$ non-isomorphic simple $\g_0$-modules from the restriction. Then it is easy to see that $(z-c_1)(z-c_2)\ldots (z-c_s)$ is in the intersection of the annihilators of these $s$ non-isomorphic simple $\g_0$-modules. But then this element is also in the annihilator of $V_{\lambda\mu}$ since $\UU(\g_0)$ is a subalgebra of $\UU(\g)$. 
\end{enumerate}
\end{proof}

\subsection{Simple tensor modules are determined by their annihilators}
\vspace{0.1cm}
As in the last section, $\g\simeq gl(\infty, \C), sl(\infty, \C), o(\infty, \C)$ or $sp(\infty, \C)$.
The first result is a particular case of the main result. However, for reasons that will become obvious later on, we include it here. 

\begin{mythm}
For pairs of partitions $(\lambda,\mu)$ and $(\lambda',\mu')$ we have $Ann_{\UU(\g)}V_{\lambda'\mu'}\subseteq Ann_{\UU(\g)}V_{\lambda\mu}$ if and only if $(\lambda,\mu)\preceq (\lambda',\mu')$. 
\end{mythm}
\begin{proof} We argue for $\g\simeq gl(\infty, \C)$ and $sl(\infty, \C)$. We omit the other cases since they are very similar to these two cases.  

Notice that if $(\lambda,\mu)\preceq (\lambda',\mu')$ then, using Theorem $4.2.5.$ iteratively, we get $SC(V_{\lambda\mu}, \g_n)\subseteq SC(V_{\lambda'\mu'}, \g_n), \forall n\in\Z_{\geq 2}$. This implies $Ann_{\UU(\g_n)}V_{\lambda'\mu'}\subseteq Ann_{\UU(\g_n)}V_{\lambda\mu}, \forall n\in\Z_{\geq 2}$. Taking union over all $n\geq 2$ yields $Ann_{\UU(\g)}V_{\lambda'\mu'}\subseteq Ann_{\UU(\g)}V_{\lambda\mu}$.  

Now, assume $Ann_{\UU(\g)}V_{\lambda'\mu'}\subseteq Ann_{\UU(\g)}V_{\lambda\mu}$. This implies that $SC(V_{\lambda\mu}, \g_n)\subseteq SC(V_{\lambda'\mu'}, \g_n), \forall n\in\Z_{\geq 2}$. Set
$$(\lambda,\mu)=(\lambda_1\geq\lambda_2\geq\ldots\geq\lambda_p,-\mu_q\geq\mu_{q-1}\geq\ldots\geq\mu_1)$$
$$(\lambda',\mu')=(\lambda_1'\geq\lambda_2'\geq\ldots\geq\lambda_r,-\mu_s'\geq\mu_{s-1}'\geq\ldots\geq\mu_1')$$
and denote by $(\lambda,\mu)_k$ and $(\lambda',\mu')_k$, the corresponding pairs of partitions but with $k$ zeros in the middle. Here, we need to point out that $(\lambda,\mu)_k$ and $(\lambda',\mu')_k$ are not pair of partitions anymore but rather decreasing sequences and highest weights of finite-dimensional modules. Following the construction of the simple tensor modules in Section $3.2.$, we have $V_{\lambda\mu}=\displaystyle\lim_{\longrightarrow}F_{n+p+q}^{(\lambda\mu)_n}$ and $V_{\lambda'\mu'}=\displaystyle\lim_{\longrightarrow}F_{n+r+s}^{(\lambda'\mu')_n}$, see for instance \cite{PStyr}. 

Taking all of the above into account, we see that for every $k\in\Z_{\geq 2}$, there exists $i\in\Z_{\geq 2}$ such that $F_{k+p+q}^{(\lambda,\mu)_k}$ is a simple direct summand of the module obtained from $F_{i+r+s}^{(\lambda',\mu')_i}$ via a number of restrictions of type $\g_{j-1}\subset\g_j$. The isomorphism class of these finite-dimensional modules is given by the sequences $(\lambda,\mu)_i$ in the case of $gl(i+p+q)$. 

In this case, we immediately obtain that $p\leq r, q\leq s$ and $\lambda_k\leq\lambda_k', k\in\{1,\ldots, p\}$, $\mu_k\leq\mu_k', k\in\{1,\ldots, q\}$. More precisely, for $i, k$ large enough, if we set $j=i+r+s-k-p-q$ to be the number of restrictions necessary to get from $F_{i+r+s}^{(\lambda',\mu')_i}$ to $F_{k+p+q}^{(\lambda,\mu)_k}$, we get $(\lambda,\mu)\in GT^{\circ j}(\{(\lambda',\mu')\})$. Thus $(\lambda,\mu)\preceq (\lambda',\mu')$.  

In the later case of $sl(i+p+q, \C)$, the isomorphism class depends on the pairwise differences $\lambda_l-\lambda_{l+1}, \mu_l-\mu_{l+1}$. Given this, notice that if we choose $k$ to be large enough then the pairwise differences for $F_{k+p+q}^{(\lambda,\mu)_k}$ are $\lambda_l-\lambda_{l+1}, l\in\{1,\ldots, p\}, \lambda_{p+1}=0$, followed by $k-1$ zeros and lastly, $\mu_l-\mu_{l+1}, l\in\{1,\ldots,q\}, \mu_{q+1}=0$. These differences have to appear, using a previous argument, in the same order in the restriction of some $F_{i+r+s}^{(\lambda',\mu')_i}$ with $i$ large enough.

Pick some $R\in\Z_{>p+q+r+s}$. Let $d_1, d_2,\ldots$ be the differences appearing in this order in $(\lambda,\mu)_k$ and $d_1', d_2', \ldots$ be the differences appearing in this order in $(\lambda',\mu')_i$. After successive restrictions of $(\lambda',\mu')_i$, notice that the nonzero consecutive differences of the pair of partitions obtained, must lie among the first $r$ such differences and the last $s$, so its number of nonzero consecutive differences does not exceed $r+s$. In the sequence $d_1,d_2,\ldots$ we have at most $p+q$ nonzero differences and a subsequence of at least $k-1$ zeros which, for $k$ large enough, is the only subsequence of consecutive differences with at least $R$ consecutive zeros. After restricting $(\lambda',\mu')_i$ to $g_{k+p+q}$, we get at most $r+s$ nonzero differences and a subsequence of at least $k+p+q-r-s$ consecutive zeros, which is the only subsequence of consecutive differences with at least $R$ consecutive zeros. 

Because of the uniqueness of the subsequences of consecutive zeros of length at least $R$, we get that, for $k$ and $i$ large enough, these subsequences of consecutive zeros must coincide. But then, taking into account that we can recover the sequence $(\lambda,\mu)_k$ from the sequence $d_1, d_2,\ldots$ and the position of the subsequence of zeros in the middle, we obtain $(\lambda,\mu)\in GT^{\circ i+r+s-k-p-q}(\{(\lambda',\mu')\})$ and thus $(\lambda,\mu)\preceq (\lambda',\mu')$. 


\end{proof}

\begin{mycor}
$Ann_{\UU(\g)} V_{\lambda\mu}=Ann_{\UU(\g)} V_{\lambda'\mu'}$ if and only if $(\lambda,\mu)=(\lambda',\mu')$. 
\end{mycor}
\begin{proof}
Notice that Theorem $4.4.1.$ does not mention the connection between the strictness of the inclusion $Ann_{\UU(\g)}V_{\lambda'\mu'}\subseteq Ann_{\UU(\g)}V_{\lambda\mu}$ and the strictness of the inequality $(\lambda,\mu)\preceq (\lambda',\mu')$. However, if we consider the double inclusion $Ann_{\UU(\g)}V_{\lambda'\mu'}\subseteq Ann_{\UU(\g)}V_{\lambda\mu}$ and $Ann_{\UU(\g)}V_{\lambda\mu}\subseteq Ann_{\UU(\g)}V_{\lambda'\mu'}$, then we obtain the conclusion as a consequence of the fact that $\preceq$ is a partial order. 
\end{proof}
\vspace{0.1cm}
\begin{myex}
Denote by $V_n$ the natural $\g_n$-module, and by $\C_n$ the trivial one-dimensional $\g_n$-module. For a $\g$-module $M$, denote by $M_{| n}$ the restriction of $M$ to $\g_n$. We have  
$$S^2V_{|n}=S^2V_n\oplus (M_1\otimes V_n) \oplus (M_2\otimes \C_n),$$
$$V_{|n}=V_n\oplus (m_3\otimes \C_n), $$
where $M_1, M_2. M_3$ are countable-dimensional trivial $\g_n$-modules. This implies $SC(V, \g_n)\subset SC(S^2V, \g_n), \forall n\in\Z_{\geq 2}$ and thus $Ann_{\UU(\g)}S^2V\subset Ann_{\UU(\g)}V$.  This inclusion corresponds to the inequality $(2,0)\succeq (1,0)$. Similarly we get the chain of inclusions 
$$\ldots\subset Ann_{\UU(\g)}S^3V\subset Ann_{\UU(\g)}S^2V\subset Ann_{\UU(\g)}V$$
which corresponds to the chain of inequalities
$$\ldots\succeq (3,0)\succeq (2,0)\succeq (1,0). $$
\end{myex}
\vspace{0.1cm}
\begin{mythm}
Let $M$ be a simple $\g$-module. For a pair of partitions $(\lambda,\mu)$, we have $Ann_{\UU(\g)}V_{\lambda,\mu}\subseteq Ann_{\UU(\g)}M$ if and only if $M\simeq V_{\lambda',\mu'}$ for some pair of partitions $(\lambda',\mu')$ and $(\lambda',\mu')\preceq (\lambda,\mu)$.  
\end{mythm}
\begin{proof} As before, we argue only for $gl(\infty, \C)$ and $sl(\infty, \C)$. In the proof, when we refer to a highest weight vector of a $\g_n$-module, we refer to a highest weight vector with respect to the Borel subalgebra $\b_n^+\subset\g_n$ defined in Section $4.2.$ 

In light of Theorem $4.4.1.$ one implication is obvious. 

Now, assume $Ann_{\UU(\g)}V_{\lambda,\mu}\subseteq Ann_{\UU(\g)}M$. Proposition $4.2.2.$ tells us that $M$ is integrable and moreover, $SC(M,\g_0)\subseteq SC(V_{\lambda,\mu}, \g_0)$ for any finite-dimensional semisimple subalgebra of $\g$. In particular, we get  
$$SC(M, \g_n)\subseteq SC(V_{\lambda,\mu}, \g_n), \mbox{ } \forall n\in\Z_{\geq 2}.\mbox{ }\mbox{ }\mbox{ }\mbox{ }\mbox{ }\mbox{ }\mbox{ }(4)$$
Pick a partition $(\lambda^1,\mu^1)=(\lambda_1^1\geq\lambda_2^1\geq\ldots\geq\lambda_r^1,-\mu_s^1\geq\ldots\geq -\mu_2^1\geq -\mu_1^1)$ such that $(\lambda^1,\mu^1)_k$ is the highest weight of one of the elements of $SC(M,\g_{2(p+q)+1})$, i.e. $k+r+s=2(p+q)+1$. Because of $(4)$. $(\lambda^1,\mu^1)_k$ comes from successive restrictions of some $F_{i+p+q}^{(\lambda,\mu)_i}$ with $i$ arbitrarily large. Note that after restricting $F_{i+p+q}^{(\lambda,\mu)_i}$ to $\g_{2(p+q)+1}$ we obtain partitions whose sequence of consecutive differences contains a subsequence of at least $p+q$ consecutive zeros and this is the only such subsequence of length at least $p+q$. We obtain that $r\leq p$, $s\leq q$ and $k\geq p+q+1$, and moreover that $(\lambda^1,\mu^1)\preceq (\lambda,\mu)$  

For any $j\in \Z_{\geq 2}$ the elements of $SC(M, \g_j)$ come from the elements of $SC(M, \g_{j+1})$ by restricting to $\g_j$. By Lemma $4.2.6.$ if $v$ is a highest weight vector of some $\g_{j+1}$-module $X$ isomorphic $F_{j+1}^{(\lambda^1,\mu^1)_{j+1-r-s}} \in SC(M, \g_{j+1})$ then it is also a highest weight vector for some $\g_j$-module $Y$ appearing in $X_{|\g_j}$, isomorphic to $F_{j}^{(\lambda^1,\mu^1)_{j-r-s}} \in SC(M, \g_{j})$. Using Theorem $4.2.5.$ $j>2(p+q)$, the multiplicity of $Y$ in $X_{|\g_{j}}$ is $1$. Indeed, in the case $\g\simeq gl(\infty, \C)$ all restrictions $\g_{n+1}\downarrow\g_n$ are multiplicity free. In the case $\g\simeq sl(\infty, \C)$, suppose for the sake of contradiction that $Y$ has multiplicity at least $2$ in $X_{|\g_j}$. Theorem $4.2.5.$ tells us that it is then possible to choose two different pairs of partitions 
$$x_1\geq x_2\geq\ldots\geq x_r\geq 0, \ldots, 0\geq -y_s\geq \ldots\geq -y_2\geq -y_1$$
$$x_1'\geq x_2'\geq\ldots\geq x_r'\geq 0, \ldots, 0\geq -y_s'\geq \ldots\geq -y_2'\geq -y_1'$$
that yield the same sequence of consecutive differences. Let these differences be $d_1,d_2,\ldots,d_{j-1}$. Since $j>2(p+q)$, we get  
$$x_1=d_1+d_2+\ldots+d_{r+1}=x_1'$$
$$x_2=d_2+d_3+\ldots+d_{r+1}=x_2'$$
$$\ldots$$
$$y_1=d_{j-1-s}+d_{j-s}+\ldots+d_{j-1}=y_1'$$
$$y_2=d_{j-1-s}+d_{j-s}+\ldots+d_{j-2}=y_2'$$
$$\ldots$$
and hence the pairs of partitions are equal, contradiction. Therefore $Y$ has multiplicity $1$ in $X_{|\g_j}$.

This discussion implies that, for a highest vector $v$ of a $\g_j$-module $X, j\in \Z_{\geq 2(p+q)+1}$ appearing in $M_{|j}$, isomorphic to $F_{j}^{(\lambda^1,\mu^1)_{j-r-s}}$, we have the following two possibilities : 
\begin{enumerate}
\item $v$ is also the highest weight vector of a $\g_{j+1}$-module $Y$ appearing in $M_{|j+1}$, isomorphic to \newline $F_{j+1}^{(\lambda^1,\mu^1)_{j+1-r-s}}$. 
\item $X$ comes from the restriction to $\g_{j}$ of a $\g_{j+1}$-module $Y$ (not necessarily unique) appearing in $M_{|j+1}$, non-isomorphic to $F_{j+1}^{(\lambda^1,\mu^1)_{j+1-r-s}}$. 
\end{enumerate}
Set $m=2(p+q)+1$. Choose a highest vector $v_1$ of a $\g_m$-module appearing in $M_{|m}$, isomorphic to $F_{m}^{(\lambda^1,\mu^1)_{m-r-s}}$. Apply the following procedure. Let $j$ be the highest index for which case $1$ applies to $v_1$. This implies $F_{j}^{(\lambda^1,\mu^1)_{j-r-s}}$ is in case $2$. Let $(\lambda_1^2\geq\lambda_2^2\geq\ldots\geq\lambda_a^2>0,\ldots,0>-\mu_b^2\geq\ldots\geq -\mu_2^2\geq -\mu_1^2)=(\lambda^2,\mu^2)_{j+1-a-b}$ be the highest weight of the simple $\g_{j+1}$-module $Y$ from case $2$. As before, we get $a\leq p$, $b\leq q$, $j+1-a-b\geq m-p-q=p+q+1$.  But notice that, by reproducing an argument given in the proof of Theorem $4.4.1.$, we easily get $(\lambda^1,\mu^1)\preceq (\lambda^2,\mu^2)$. Moreover, because $F_{j+1}^{(\lambda^2,\mu^2)_{j+1-a-b}}$ must come by successive restrictions from $F_{i+p+q}^{(\lambda,\mu)_i}$ with $i$ arbitrarily large, we get $(\lambda^2,\mu^2)\preceq (\lambda,\mu)$. Taking into account that $F_{j+1}^{(\lambda^2,\mu^2)_{j+1-a-b}}$ is not isomorphic to $F_{j+1}^{(\lambda^1,\mu^1)_{j+1-r-s}}$, we get $(\lambda^1,\mu^1)\prec (\lambda^2,\mu^2)\preceq (\lambda\mu)$. 



Now we pick $v_2$ to be the highest weight vector of the $g_{j+1}$-module $Y$ and we continue the procedure inductively. We obtain a chain of strict inequalities $(\lambda^1,\mu^1)\prec (\lambda^2,\mu^2)\prec\ldots\prec (\lambda^s,\mu^s)\preceq (\lambda,\mu), s\in\Z_{\geq 2}$ of pair of partitions. Notice that any consecutive difference that appears in some pair of partitions in $GT^{\circ i}(\{(\lambda,\mu)\})$ is bounded by $0$ and $\lambda_1+\mu_1$ and there are at most $p+q$ nonzero such differences. This implies that there are at most $(\lambda_1+\mu_1)^{p+q}$ different pairs of partitions $(\lambda',\mu')$ with $(\lambda',\mu')\preceq (\lambda,\mu)$. Given this and the chain of inequalities of pairs of partitions that we previously obtained, it follows that from some point on, the procedure will only be in case $1$. 

Suppose $v_s$ is the last highest vector chosen after the procedure stabilizes in case $1$.  Since $v_s$ is a highest weight vector with respect to all $\b_n, n\in\Z_{\geq m}$, we get $\n_+ v_s=\{0\}$. Moreover, Lemma $4.2.6.$ and $V_{\lambda^s\mu^s}=\displaystyle\lim_{\longrightarrow}F_{n+p+q}^{(\lambda^s,\mu^s)_n}$, imply that $v_s$ generates a submodule of $M$ isomorphic to $V_{\lambda^s\mu^s}$. Given the simplicity of $M$, we obtain that $M\simeq V_{\lambda^s\mu^s}$ and by Theorem $4.4.1.$ the conclusion follows. Thus we are done. 
\end{proof}
\vspace{0.1cm}
\begin{mycor}
If $M$ is a simple $\g$-module (not necessarily integrable) such that $Ann_{\UU(\g)}M=Ann_{\UU(\g)}V_{\lambda\mu}$ for some pair of partitions $(\lambda,\mu)$ then $M\simeq V_{\lambda\mu}$. 
\end{mycor}
\vspace{0.1cm}

\vspace{2cm}
I would like to thank my advisor Professor Ivan Penkov for all the help, support and numerous advices. This text would not have taken form without him. I would also like to thank my colleague Alexei Pethukov for pointing me to relevant references. 

\vspace{2cm}
Alexandru Sava,

School of Engineering and Science, Department of Mathematics,

Jacobs University Bremen. 

a.sava@jacobs-university.de

\bibliographystyle{plain}

 \end{document}